\documentclass{article}

\usepackage[margin=3cm]{geometry}
\usepackage{amsmath,amssymb,amsbsy,graphicx,bbm}
\usepackage{amsthm}
\usepackage[round]{natbib}
\usepackage[T1]{fontenc}
\usepackage{setspace}
\usepackage{enumitem}
\usepackage[pdftex,dvipsnames]{xcolor}
\usepackage{authblk}
\usepackage{subcaption}

\newtheorem{proposition}{Proposition}
\newtheorem*{proposition*}{Proposition}

\newtheorem{definition}{Definition}
\newtheorem*{definition*}{Definition}
\newtheorem{theorem}{Theorem}
\newtheorem*{theorem*}{Theorem}
\newtheorem{corollary}{Corollary}
\newtheorem{lemma}{Lemma}

\newtheorem{example}{Example}
\newtheorem*{example*}{Example}

\renewcommand{\P}{\mathbb{P}}
\newcommand{\R}{\mathbb{R}}
\newcommand{\N}{\mathbb{N}}
\newcommand{\Z}{\mathbb{Z}}
\newcommand{\E}{\mathbb{E}}
\newcommand{\I}{\mathbb{I}}

\newcommand{\mF}{\mathcal{F}}
\newcommand{\mG}{\mathcal{G}}
\newcommand{\mH}{\mathcal{H}}
\newcommand{\mM}{\mathcal{P}} 
\newcommand{\mN}{\mathcal{N}}
\newcommand{\mS}{\mathcal{S}}
\newcommand{\mV}{\mathcal{V}}

\newcommand{\argmin}{\operatorname*{\arg\min}}
\newcommand{\Argmin}{\operatorname*{Arg\min}}
\newcommand{\1}{{\bf 1}}

\newcommand{\eps}{\varepsilon}
\newcommand{\row}{{\rm row}}
\newcommand{\col}{{\rm col}}
\newcommand{\rk}{{\rm rk}}
\newcommand{\codim}{{\rm codim}}
\newcommand{\conv}{{\rm conv}}
\newcommand{\supp}{{\rm supp}}
\newcommand{\sign}{{\rm sign}}
\newcommand{\patt}{{\rm patt}}
\newcommand{\round}{{\rm round}}
\newcommand{\stab}{{\rm stab}}
\newcommand{\orb}{{\rm orb}}
\newcommand{\relint}{{\rm relint}}
\newcommand{\aff}{{\rm aff}}
\renewcommand{\vert}{{\rm vert}}
\newcommand{\betaOLS}{\hat\beta^{\rm ols}}
\newcommand{\betaRIDGE}{\hat\beta^{\rm ridge}}
\newcommand{\betaLASSO}{\hat\beta^{\rm lasso}}
\newcommand{\betaBP}{\hat\beta^{\rm bp}}
\newcommand{\betaSLOPE}{\hat\beta^{\rm slope}}

\newcommand{\Spen}{S_{X,\|.\|}}
\newcommand{\SpenZ}{S_{Z,\|.\|}}
\newcommand{\Ssup}{S_{X,\|.\|_\infty}}

\newcommand{\Sbp}{S_{X,{\rm bp}}}

\newcommand{\SlassoL}{S_{X,\lambda\|.\|_1}}
\newcommand{\Sslope}{S_{X,\|.\|_w}}
\newcommand{\dnorm}{\partial_{\|.\|}}
\newcommand{\dnormx}[1]{\partial_{\|.\|_{#1}}}

\renewcommand{\emptyset}{\varnothing}

\setlist[enumerate]{label=\arabic*)}

\colorlet{shadecolor}{gray!20}

\begin{document}

\onehalfspacing 

\author[1]{Ulrike Schneider 
} 
\affil[1]{TU Wien}

\author[2]{Patrick  Tardivel}
\affil[2]{University of Wroc\l{}aw and University of Burgundy}

\title{The Geometry of Uniqueness, Sparsity and Clustering in Penalized Estimation}

\date{}

\maketitle

\begin{abstract}
We provide a necessary and sufficient condition for the uniqueness of
penalized least-squares estimators whose penalty term is given by a
norm with a polytope unit ball, covering a wide range of methods
including SLOPE, PACS, fused, clustered and classical LASSO as well as
the related method of basis pursuit. We consider a strong type of
uniqueness that is relevant for statistical problems. The uniqueness
condition is geometric and involves how the row span of the design
matrix intersects the faces of the dual norm unit ball, which for
SLOPE is given by the signed permutahedron. Further considerations
based this condition also allow to derive results on sparsity and
clustering features. In particular, we define the notion of a SLOPE
pattern to describe both sparsity and clustering properties of this
method and also provide a geometric characterization of accessible
SLOPE patterns.
\end{abstract}

\medskip

\noindent {\bf Keywords:} penalized estimation, SLOPE, uniqueness,
sparsity, clustering, regularization, geometry, polytope.


\smallskip

\noindent {\bf MSC 2020:} Primary 62-08; Secondary 52B12.

\section{Introduction} \label{sec:intro}

The linear regression model $Y = X\beta + \eps,$ where $X \in \R^{n
\times p}$ is a fixed matrix, $\beta\in \R^p$ is an unknown parameter
vector, and $\eps$ is a centered random error term in $\R^n$, plays a
central role in statistics. When $\ker(X) = \{0\}$, the ordinary
least-squares estimator $\betaOLS = (X'X)^{-1}X'Y$, which minimizes
the residual sum of squares $\|Y - Xb\|_2^2$ with respect to $b \in
\R^p$, is the usual estimator of $\beta$. In high dimensions, when $p
> n$, and thus $\ker(X) \neq \{0\}$, the ordinary least squares
estimator is no longer well-defined, as then the function $b\in \R^p
\mapsto \|Y-Xb\|_2^2$ does not have a unique minimizer.

In this case, typically, a penalty term is added to the residual sum
of squares to provide an alternative to ordinary least-squares
estimation. In some cases, also the minimizer of the penalized
least-squares optimization problem is not unique. Since $Y$ is a
random vector and the induced stochastic properties on the minimizer
are often the object of study in a statistical framework, it is
relevant to consider a strong type of uniqueness: uniqueness for a
given $X$ that holds for all realizations\footnote{Certain results in
the literature \citep{ZhangEtAl15,Gilbert17,MousaviShen19} provide a
criterion for the uniqueness of a given minimizer. These results
naturally differ strongly from the ones in the present article as they
deal with a weaker notion of uniqueness.} of $Y$ in $\R^n$. In this
paper, we provide a necessary and sufficient condition for uniqueness
for a wide class of penalties based on a geometric criterion, as well
as for the related methods of basis pursuit. Moreover, the geometry
involved in this condition also yields results for model selection and
pattern recovery, i.e., sparsity and related clustering properties,
which we investigate for SLOPE in particular.

\subsection{Penalized least-squares estimators and uniqueness} \label{subsec:intro_penest}

The Ridge estimator, minimizing the function  $b \in \R^p \mapsto
\frac{1}{2} \left\|Y-Xb\right\|_2^2 + \lambda \|b\|^2_2$, where
$\lambda > 0$ is a tuning parameter, was the first penalized estimator
to appear in the statistics literature
\citep{HoerlKennard70,GolubEtAl79}. Due to the strict convexity of the
function $b \mapsto \|b\|_2^2$, the minimizer is always unique and
given by $\betaRIDGE = (X'X + \lambda \I_p)^{-1}X'Y$. This estimator
is not sparse, meaning that it does not set components equal to zero
almost surely. Especially when $p$ is large, this can make the
estimator more difficult to interpret compared to other methods such
as LASSO or SLOPE, which do exhibit sparsity and are described in the
following.


The Least Absolute Shrinkage and Selection Operator or LASSO
\citep{ChenDonoho94,AllineyRuzinsky94,Tibshirani96} is the
$\ell_1$-penalized least-squares estimator defined as
$$
\betaLASSO = \argmin_{b \in \R^p} \frac{1}{2} \left\|Y-Xb\right\|_2^2
+ \lambda\|b\|_1, \text{ \rm where } \lambda > 0.
$$
When $\ker(X) = \{0\}$, the function $b \in \R^p \mapsto \left\|Y -
Xb\right\|_2^2$ is strictly convex, immediately implying the
uniqueness of the LASSO minimizer. In high dimensions, $\ker(X) \neq
\{0\}$ and the function $b \in \R^p \mapsto \left\|Y - Xb\right\|_2^2$
is not strictly convex, thus uniqueness of $\betaLASSO$ is not
guaranteed. A geometric description of the set of LASSO minimizers,
relevant when non-uniqueness occurs, is given in
\cite{DupuisVaiter19TR}. A sufficient condition for uniqueness of the
estimator for all $Y \in \R^n$ is for the columns of the design matrix
$X$ to be in general position. This was first outlined by
\cite{RossetEtAl04} and later investigated by \cite{Tibshirani13} and
\cite{AliTibshirani19}. Recently, this condition was relaxed by
\cite{EwaldSchneider20} to a geometric criterion that is both
sufficient and necessary and which is generalized for a wide class of
possible penalty terms in the present paper.


A strongly related procedure is basis pursuit, which first appeared in
compressed sensing \citep{ChenDonoho94} and is defined as
$$
\betaBP = \argmin \|b\|_1 \text{ \rm subject to }  Y = Xb,
$$
provided that $Y \in \col(X)$. In the noiseless case, this method
allows to recover a sparse vector $\beta$ \citep[see
e.g.][]{CandesEtAl06,CohenEtAl09}. In the noisy case, when $\eps$ is
no longer zero, the basis pursuit estimator can be viewed as the LASSO
when the tuning parameter $\lambda > 0$ becomes infinitely small
\cite[Lemma 3.6]{Dossal12}\footnote{This reference focuses on
necessary and sufficient conditions to uniquely recover a given $b_0$
from $y=Xb_0$ (in our notation), which is a different type of
uniqueness than we consider.}. Naturally, basis pursuit shares  a lot
of properties with the LASSO estimator. For example, general position
of the columns of the design matrix $X$ is also a sufficient condition
for uniqueness of $\betaBP$ for all $Y \in \R^n$ \citep[see
e.g.][]{Dossal12}. However, to the best of our knowledge, a necessary
and sufficient condition for this type of uniqueness has previously
been unknown.


Our results also cover Sorted L-One Penalized Estimation or SLOPE
\citep{BogdanEtAl15,NegrinhoMartins14,ZengFigueiredo14}, which is the penalized
estimator given by
$$
\betaSLOPE = \argmin_{b \in \R^p} \frac{1}{2} \left\|Y - Xb\right\|_2^2 
+ \sum_{j=1}^p w_j|b|_{(j)},
$$ 
where $w_1 > 0$, $w_1 \geq \dots \geq w_p \geq 0$, and $|b|_{(1)} \geq
\dots \geq |b|_{(p)}$. Note that the penalty term gives rise to the
so-called sorted-$\ell_1$-norm. A special case of this estimator, the Octagonal
Shrinkage and Clustering Algorithm for Regression or OSCAR, has
already been introduced in \cite{BondellReich08}. The SLOPE estimator
is well-defined once the corresponding minimizer is unique and,
similarly to the LASSO, uniqueness is obvious when $\ker(X) = \{0\}$.
However, in contrast to the LASSO, no condition guaranteeing
uniqueness has previously been established.

\subsection{Uniqueness and polytope unit balls}
\label{subsec:intro_unique}

In this paper, we study the problem of uniqueness of penalized
estimators in a general setting, where the penalty term is not
restricted the $\ell_1$- or the sorted-$\ell_1$-norm. We describe the framework we
consider in the following. Let $X \in \R^{n\times p}$, $y \in \R^n$,
and $\|.\|$ be a norm on $\R^p$. Consider the solution set
$S_{X,\|.\|}(y)$ to the penalized least-squares problem
$$
\Spen(y) = \Argmin_{b \in \R^p} \frac{1}{2} \left\|y - Xb\right\|_2^2 + \|b\|.
$$
Note that $\Spen(y)$ is non-empty since the function $b \in \R^p
\mapsto \frac{1}{2} \left\|y-Xb\right\|_2^2 + \|b\|$ is continuous and
unbounded when $\|b\|$ becomes large. The penalty term may include a
positive tuning parameter which can be viewed as part of the norm, for
instance $\|.\| = \lambda \|.\|_1$ for the LASSO estimator. When
$\|.\|$ is a norm for which $\|b + \tilde b\| = \|b\|+\|\tilde b\|$
holds if and only if $b = t \tilde b$ where $t \geq
0$\footnote{Typically, $b$ and $\tilde b$ are not orthogonal, thus the
equality in the triangular inequality does not coincide with the
decomposability property described in \cite{NegahbanEtAl12}.}, such as
the $\ell_2$-norm, then $\Spen(y)$ is a singleton for all $y \in \R^n$
and for all $X \in \R^{n \times p}$. This statement is a
straightforward consequence of the following facts. When $\hat\beta,
\tilde\beta \in \Spen(y)$ we have

\begin{itemize}

\item[i)] $X\hat\beta = X\tilde\beta$ (see
Lemma~\ref{lem:fitted-values} in the appendix).

\item[ii)] Since $(\hat\beta + \tilde\beta)/2 \in \Spen(y)$ also,
$\|(\hat\beta + \tilde\beta)/2\| = \|\hat\beta\| = \|\tilde\beta\| =
(\|\hat\beta\| + \|\tilde\beta\|)/2$ follows.

\end{itemize}

Geometrically, such a norm $\|.\|$ possesses a unit ball $\{x \in \R^p
: \|x\| = 1\}$ with no edges. Subsequently, the problem of uniqueness
is only relevant when the unit ball of the norm under consideration
contains an edge. More concretely, we restrict our attention to norms
for which the unit ball $B = \{x \in \R^p : \|x\| \leq 1\}$ is given
by a polytope. Note that this is the case for the $\ell_1$-norm, the
$\ell_\infty$-norm, and the sorted-$\ell_1$-norm. Our results also
cover the fused LASSO \citep{TibshiraniEtAl05}, the Pairwise Absolute
Clustering and Sparsity (PACS) procedure \citep{SharmaEtAl13}, the
clustered LASSO \citep{She10}, or methods with a mixed
$\ell_1$,$\ell_\infty$-norm penalty term
\citep{NegahbanWainwright08,BachEtAl12}.

\subsection{Sparsity and clustering: accessible patterns and sign
estimation} \label{subsec:intro_accpatterns}

As mentioned above, the LASSO estimator is a sparse method that
generally sets components equal to zero with positive probability,
entailing that the estimator also performs so-called model or variable
selection. In fact, when $p > n$ and the solution is unique,
$\betaLASSO$ contains at least $p - n$ zero components. Instigated by
this property, an abundant literature has arisen to deal with the
recovery of the location of the non-null components of $\beta$, or,
more specifically, the recovery of the sign vector of $\beta$
\citep{Zou06,ZhaoYu06,Wainwright09}.

A necessary condition for the recovery of $\sign(\beta)$ is for this
vector to be accessible by the LASSO, i.e., for a fixed $\lambda > 0$,
there has to exist $Y \in \R^n$ for which $\sign(\betaLASSO) =
\sign(\beta)$. Otherwise, $\P(\sign(\betaLASSO) = \sign(\beta)) = 0$,
and recovery is clearly impossible. A geometrical characterization of
accessible sign vectors is given in \cite{SepehriHarris17} under the
assumption of uniqueness of LASSO solutions.
In the appendix, we provide a geometrical characterization of
accessible sign vectors for both basis pursuit and LASSO without a
uniqueness assumption.


In the OSCAR procedure mentioned in Section~\ref{subsec:intro_penest},
the letter ``C'' stands for ``Clustering'', referring to the fact that
some components of this estimator can be equal in absolute value. This
property can be illustrated for OSCAR -- as well as the more general
SLOPE method -- by drawing the elliptic contour lines of the residual
sum of squares $b \mapsto \|Y - Xb\|_2^2$ (when $\ker(X)=\{0\}$)
together with balls of the sorted-$\ell_1$-norm\footnote{See, e.g.,
Figure~2 in \cite{BondellReich08} or Figure~1 in
\cite{ZengFigueiredo14}.}. This clustering property can also be
deduced from the explicit expressions of SLOPE one obtains for the
case where the columns of $X$ are orthogonal \citep{TardivelEtAl20,
DupuisTardivel22}. We show that the clustering phenomenon also holds
in the general case by using our geometric approach. This feature of
SLOPE -- which is not shared by the LASSO -- has, of course, been
known in practice and may be of particular relevance in certain
applications \citep{FigueiredoNowak16,KremerEtAl20,KremerEtAl22}.

With a similar objective as the articles written a decade ago on
support or sign recovery by LASSO, there are now several papers
dealing with pattern recovery by SLOPE where the notion of SLOPE
pattern is a central concept \citep{SkalskiEtAl22TR,BogdanEtAl22TR}.
In the present article, we show how our geometric approach can be used
to provide a characterization of the clusters induced by SLOPE.

\subsection{Related geometrical works}

Most articles providing geometric properties in the context of
penalized estimation treat the LASSO. \cite{TibshiraniTaylor12} show
that the LASSO residual $Y - X\betaLASSO$ is the projection of $Y$
onto the so-called LASSO null polyhedron $\{z \in \R^n: \|X'z\|_\infty
\leq \lambda\}$. From this result, the authors derive an explicit
formula for the Stein's unbiased risk estimate that provides an
unbiased estimator for $\E(\|X\betaLASSO - X\beta\|_2^2)$. This
geometric result also lays the groundwork for selective inference
\citep{LeeEtAl16}, for deriving screening procedures
\citep{ElGhaouiEtAl12,WangEtAl13}, and to describe the accessible
LASSO patterns in \cite{SepehriHarris17}.
For basis pursuit, geometrical considerations focus on dealing with
the $\ell_1$-recovery in the noiseless case and are aimed at deriving the
phase-transition curve \citep{DonohoTanner09b}.
The  recent article of \cite{Minami20} generalizes some results of
\cite{TibshiraniTaylor12} to SLOPE and shows that the number of non-null clusters (the quantity $\|\patt(\betaSLOPE)\|_\infty$ in our
article) appears in the Stein's unbiased risk estimate for SLOPE
estimator. For the sake of completeness we mention that in the present paper, we
provide a convex null set in Proposition~\ref{prop:convexnullset} that
generalizes the concept of the LASSO null polyhedron to all
norm-penalized least-squares estimators, where the projection of $Y$
onto this set yields the estimation residuals.

\subsection{Notation and structure} \label{subsec:intro_notation}

To conclude this section, we introduce the notation used throughout
this article. We denote the set $\{1,\dots,k\}$ by $[k]$ and use $|I|$ for the
cardinality of a set $I$. The set $\mS_p$ contains all permutations on
the set $[p]$. For a matrix $A$, the symbols $\col(A)$ and $\row(A)$
stand for the column and row space of $A$, respectively, whereas
$\conv(A)$ represents the convex hull of the columns of $A$. As used
in previous sections already, for a number $t$, $\sign(t)$ is given by
$1,-1$, or $0$ if $t > 0, t < 0$, or $t = 0$, respectively. For a
vector $x$, $\sign(x)$ is the vector containing the signs of the
components of $x$. Finally, the symbols $\|.\|_1$, $\|.\|_2$,
$\|.\|_\infty$, and $\|.\|_w$ represent the $\ell_1$-, $\ell_2$-,
supremum, and the sorted-$\ell_1$-norm, respectively.

The remainder of this article is organized as follows.
Section~\ref{sec:unique} contains the main theorem of uniqueness for
penalized least-squares estimators, as well as the analogous necessary
and sufficient uniqueness condition for basis pursuit. In
Section~\ref{sec:sparse}, we investigate the pattern selection
properties related to the geometric condition introduced in
Section~\ref{sec:unique} for SLOPE, including a characterization of
the SLOPE's clustering property. This section also contains a general
result on the convex null set for norm-penalized least-squares
estimation.  Appendix~\ref{app:LASSO_BP} illustrates what our results
entail for LASSO and BP. All proofs are relegated to
Appendix~\ref{app:proofs}, which also contains a remainder of basic
facts of subdifferentials and polytopes.

\section{A necessary and sufficient condition for uniqueness of
penalized problems} \label{sec:unique}

We start by providing the framework for the theorem on uniqueness of
penalized least-squares minimization problems. For a norm $\|.\|$ on
$\R^p$, the dual norm $\|.\|^*$ is defined by
$$
\|x\|^* = \sup_{s \in \R^p: \|s\| \leq 1} s'x.
$$
If the unit ball $B = \{x \in \R^p: \|x\| \leq 1\}$ is of polytope
shape, the dual of $B$ given by $B^* = \{x \in \R^p: \|x\|^* \leq
1\}$, the unit ball of the dual norm, is, again, a polytope. In this
case, the penalty term is not differentiable everywhere and there is a
strong connection between the subdifferentials $\dnorm(.)$ of the norm
$\|.\|$ and the faces of the polytope $B^*$. The precise association
is detailed in Appendices~\ref{app:subdiffs}-\ref{app:subdiff-norm}
and this connection provides the basis for the main theorem.

\begin{theorem}[Necessary and sufficient condition for uniqueness] \label{thm:uni-pen} 
Let $X \in \R^{n \times p}$ and let $\|.\|$ be a norm on $\R^p$ whose
unit ball $B$ is given by a polytope. Consider the penalized
optimization problem
\begin{equation} \label{eq:pen}
\Spen(y) = \Argmin_{b \in \R^p} \frac{1}{2} \|y - Xb\|^2 + \|b\|,
\end{equation}
where $y \in \R^n$. Let $B^*$ denote the unit ball of the dual norm
$\|.\|^*$. There exists $y \in \R^n$ with $|\Spen(y)| > 1$ if and only
if $\row(X)$ intersects a face of the dual unit ball $B^*$ whose
codimension is larger than ${\rm rk}(X)$.
\end{theorem}

It should be noted that also vertices are faces (of dimension zero and
codimension $p$), as is the polytope itself (of dimension $p$ and
codimension zero), a precise definition for faces is given in the
appendix.

As mentioned in the introduction, the notion of uniqueness considered
in Theorem~\ref{thm:uni-pen} is strong in the sense that it guarantees
uniqueness for a given design matrix $X$ \emph{for all} values $y \in
\R^n$. This concept of uniqueness is beneficial when studying the
stochastic properties of the minimizer in a statistical framework, as
then $y$ varies and a criterion independent of $y$ is desirable. Also
note that we make no assumptions on $X$.

If the norm $\|.\|$ involves a tuning parameter $\lambda$, the
uniqueness of the corresponding penalized problem does not depend on
the particular choice of $\lambda$. The parameter simply scales $B$
and subsequently $B^*$ and does not affect which faces are intersected
by the vector space $\row(X)$.

Theorem~\ref{thm:uni-pen} generalizes Theorem~14 given in
\cite{EwaldSchneider20}, which provides a necessary and sufficient
condition for the uniqueness of the LASSO minimizer: All LASSO
solutions are unique if and only if $\row(X)$ only intersects faces of
the unit cube $[-1,1]^p$ whose codimension is less than or equal to
$\rk(X)$. Note that the unit cube is, indeed, the corresponding dual
to the unit ball of the $\ell_1$-norm. 

\begin{example}
We illustrate the criterion from Theorem~\ref{thm:uni-pen} for $\|.\|
= \|.\|_\infty$, the supremum norm, in Figure~\ref{fig:sup_norm}. Let
$X = (1 \;\; 0)$. The unit dual ball $B^*$ is given by the unit
cross-polytope $\conv\{\pm(1,0)',\pm(0,1)'\}$ and we have $\rk(X) =
1$. Clearly, the vertex $(1,0)'$ with codimension $p - 0 = 2 > 1 =
\rk(X)$ intersects $\row(X)$, so that one can pick  $y \in \R$ for
which the set of minimizers $\Ssup(y)$ is not a singleton. In
Figure~\ref{fig:sup_norm}(\subref{subfig:sup_norm-non_unique}), we
illustrate this fact for $\Ssup(2)$.

Also consider $X = (1 \;\; 1)$. Because $\row(X)$ does not intersect
any vertex of $\conv\{\pm(1,0)',\pm(0,1)'\}$, the solution set
$\Ssup(y)$ is always a singleton. In
Figure~\ref{fig:sup_norm}(\subref{subfig:sup_norm-unique}), we
illustrate this fact for $\Ssup(2)$.

\begin{figure}[htbp]

\centering

\begin{subfigure}{\textwidth}
\includegraphics[width=\textwidth]{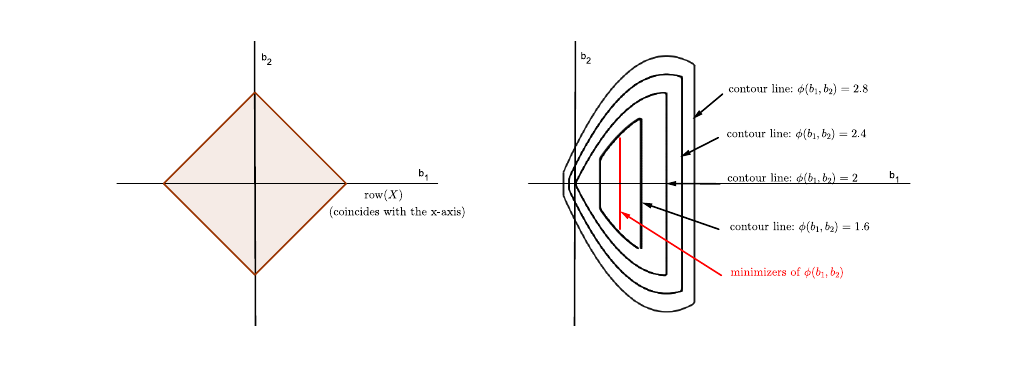}
\subcaption{\label{subfig:sup_norm-non_unique} Let $X = (1 \;\; 0)$.
On the left-hand side, we see that $\row(X)$ intersects a vertex of
the cross-polytope whose codimension is 2 and thus is larger than
$\rk(X) = 1$. Therefore, there exists $y \in \R$ for which $\Ssup(y)$
is not a singleton. On the right-hand side, the contour lines of the
objective function $\phi(b_1,b_2) = 0.5(2-b_1)^2+ \max\{|b_1|,|b_2|\}$
show that the set $\Ssup(2)$ (in red), indeed, contains infinitely many
points.}
\end{subfigure}

\begin{subfigure}{\textwidth}
\includegraphics[width=\textwidth]{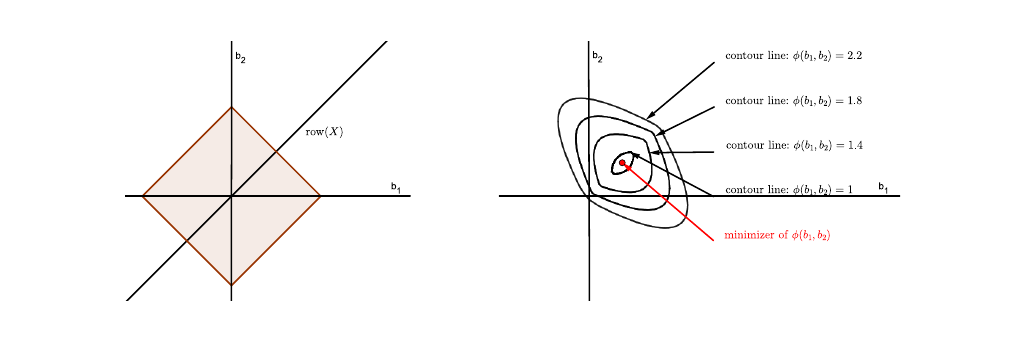} 
\subcaption{\label{subfig:sup_norm-unique}Let $X = (1 \;\; 1)$. On the
left-hand side, we see that $\row(X)$ does not intersect any face of
the cross-polytope whose codimension is larger than $\rk(X) = 1$ (such
faces are the vertices in this example). Therefore, the set $\Ssup(y)$
is a singleton for all $y \in \R$. On the right-hand side, the contour
lines of the objective function $\phi(b_1,b_2) = 0.5(2-b_1)^2 +
\max\{|b_1|,|b_2|\}$ show that the set $\Ssup(2)$ (in red) does,
indeed, only contain a single point.}
\end{subfigure}

\caption{\label{fig:sup_norm} Illustration of
Theorem~\ref{thm:uni-pen} for the supremum norm.}

\end{figure}

\end{example}

\subsection{The related problem of basis pursuit} \label{subsec:unique_BP}

As mentioned before, the methods of LASSO and basis pursuit (BP) are
closely related, as the BP problem can be thought of a LASSO problem
with vanishing tuning parameter. More concretely, the setting for BP
is the following. Let $X \in \R^{n \times p}$ and let $y \in \col(X)$.
The set $\Sbp(y)$ of BP minimizers is defined as
$$
\Sbp(y) = \Argmin \|b\|_1 \text{ subject to } Xb = y.
$$
The following theorem shows that, indeed, as BP is a limiting case of
the LASSO, the corresponding uniqueness condition -- which is
independent of the choice of tuning parameter as discussed above --
carries over to the BP problem.

\begin{theorem} \label{thm:uni-BP} 
Let $X \in \R^{n \times p}$. There exists $y \in \col(X)$ for which
$|\Sbp(y)| > 1$ if and only if $\row(X)$ intersects a face of the unit
cube $[-1,1]^p$ whose codimension is larger than $\rk(X)$.
\end{theorem}

We illustrate Theorem~\ref{thm:uni-BP} in
Figures~\ref{fig:BP}(\subref{subfig:BP-non_unique}) and
\ref{fig:BP}(\subref{subfig:BP-unique}).

\begin{figure}[htbp]

\centering

\begin{subfigure}{\textwidth}
\includegraphics[width=\textwidth]{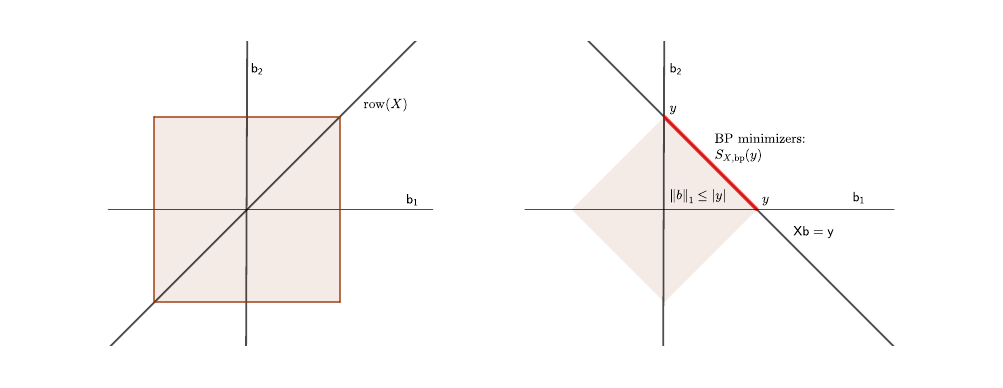}
\subcaption{\label{subfig:BP-non_unique} Let $X = (1 \;\; 1)$. On the
left-hand side, we see that $\row(X)$ intersects a face of the unit
square whose codimension 2 is larger than $\rk(X) = 1$ (which are the
vertices in this example). Therefore, by Theorem \ref{thm:uni-BP},
there exists $y \in \R$ for which the BP minimizer is not unique. The
right-hand side illustrates that, indeed, for an arbitrary $y \in
\R\setminus\{0\}$, the set $\Sbp(y)$ (the red segment) is not a
singleton.}
\end{subfigure}

\begin{subfigure}{\textwidth}
\includegraphics[width=\textwidth]{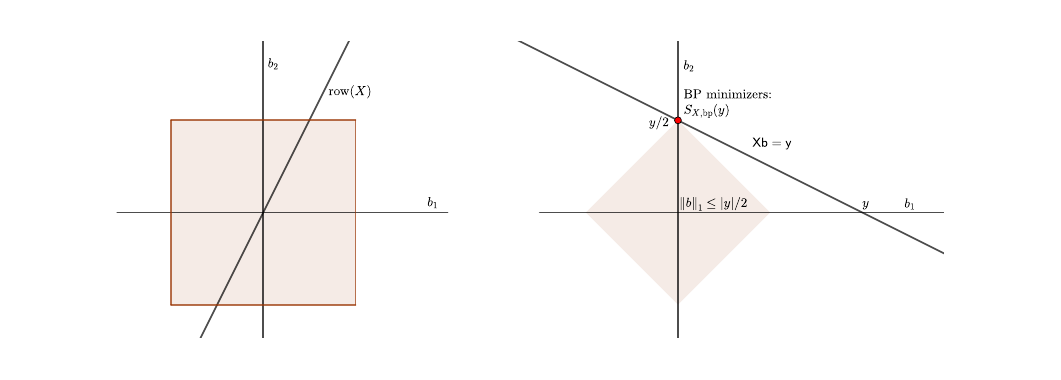}
\subcaption{\label{subfig:BP-unique}  Let $X = (1 \;\;
2)$. On the left-hand side, we see that $\row(X)$ does not intersect
any face of the unit square whose codimension is larger than $\rk(X) =
1$ (which are the vertices in this example). Therefore, by
Theorem~\ref{thm:uni-BP}, the BP minimizer is unique for all $y \in
\R$. The right-hand side illustrates that for an arbitrary $y \in \R$,
the set $\Sbp(y)$ (in red) is, indeed, a singleton.}
\end{subfigure}

\caption{\label{fig:BP} Illustration of Theorem~\ref{thm:uni-BP}.}

\end{figure}

In the following proposition, we show that the necessary and
sufficient condition given in Theorem~\ref{thm:uni-pen} and therefore
also the one given in Theorem~\ref{thm:uni-BP} is weak. More
precisely, we establish that the set of $X \in \R^{n \times p}$ for
which the necessary and sufficient condition given in
Theorem~\ref{thm:uni-pen} does not hold, is negligible with respect to
the Lebesgue measure.

\begin{proposition}
\label{prop:cond-weak} Let $\mu$ be the Lebesgue measure on
$\R^{n\times p}$ and let $\|.\|$ be a norm on $\R^p$ whose unit ball
is given by a polytope. The following equality holds
$$
\mu\left(\left\{X\in \R^{n\times p} : \exists y \in \R^n \text{ with }
|\Spen(y)| > 1\right\}\right) = 0.
$$
\end{proposition}

The following corollary is then straightforward given the fact that
the LASSO, which is covered by Theorem~\ref{thm:uni-pen}, and BP share
the same characterization for uniqueness.

\begin{corollary} \label{cor:cond-weak-BP}
Let $\mu$ be the Lebesgue measure on $\R^{n\times p}$, then the
following equality holds
$$
\mu\left(\left\{X\in \R^{n\times p} : \exists y \in \R^n \text{ with }
|\Sbp(y)| > 1 \right\}\right) = 0.
$$
\end{corollary}
By taking the appropriate norms in Proposition~\ref{prop:cond-weak},
and by Corollary~\ref{cor:cond-weak-BP}, one may deduce that the
necessary and sufficient conditions for uniqueness of SLOPE, PACS,
fused, clustered and classical LASSO are weak. However, one should be
aware that Proposition~\ref{prop:cond-weak} does not mean that this
condition always occurs in practice! 
For example, for BP (or LASSO), when $p > n$ and $X \in [-1,1]^{n
\times p}$ is a matrix having  a row  with at least $n+1$ elements in
$\{-1,1\}$ then, one can  pick $y \in \col(X)$ for which the set of
minimizers $\Sbp(y)$ is not a singleton (or, for any $\lambda > 0$,
one can pick $y \in \R^n$ for which the set of minimizers
$\SlassoL(y)$ is not a singleton). Matrices having entries in
$\{-1,1\}$ appear in several theoretical works, such as
\cite{Rauhut10} and \cite{TardivelEtAl18}, and are used for
applications in radar and wireless communication \citep[see
e.g.][]{Romberg09,HauptEtAl10}. Moreover, \cite{DupuisVaiter19TR}
recently illustrated that the matrix $X \in [-1,1]^{5000\times 6000}$,
having most entries in $\{-1,1\}$, and the vector $y \in \R^{5000}$
provided by the dataset ``gisette'' give a set of minimizers
$\SlassoL(y)$ which is not a singleton once $\lambda > 0$ is small
enough.

\section{Pattern selection properties} \label{sec:sparse}

The geometric considerations around Theorems~\ref{thm:uni-pen} and
\ref{thm:uni-BP} can also provide insights on the pattern selection
aspects of the method under consideration. The keystone is to
associate a pattern with a face of the polytope $B^*$, the unit ball
of the dual norm. For LASSO and BP, we exploit the fact that each face
of the unit cube corresponds to a sign vector and show that the faces
intersected by the row span of $X$ provide the accessible sign vectors
for these estimators, see Appendix~\ref{app:LASSO_BP}. We
take a similar, but more sophisticated approach for SLOPE in Section
\ref{subsec:sparse-SLOPE} where the patterns we consider also carry
information about the clustering phenomenon of the method.

In Section~\ref{subsec:sparse_nullpolytope}, we take a different angle
and characterize the SLOPE null polyhedron and its connection to the
sparsity and clustering property of this method. For the LASSO, it is
known that the estimation residuals are the projection of $y$ onto the
LASSO null polyhedron. We also further generalize this fact to arbitrary
norm-penalized least-squares estimation.


\subsection{Accessible patterns for SLOPE} \label{subsec:sparse-SLOPE}

We now turn to accessible patterns for SLOPE, whose norm is given by
$\|b\|_w = \sum_{j=1}^p w_j |b|_{(j)}$, where $|b|_{(1)} \geq \dots
\geq |b|_{(p)}$, as introduced before. For the remainder of
Section~\ref{sec:sparse}, we assume that the weight vector $w$ of the
satisfies
$$
w_1 > \dots > w_p > 0,
$$
i.e., that all components non-zero and strictly decreasing. (This
assumption is not needed for applying  Theorem~\ref{thm:uni-pen} to
SLOPE, since $w_1 > 0$ and decreasing components are sufficient for
$\|.\|_w$ to be a norm.) We introduce a more sophisticated notion of a
``pattern'' chosen by SLOPE compared to sign vectors that can account
for the clustering property which is not shared by LASSO or BP.

\begin{definition} \label{def:SLOPE-patterns}
We say that a vector $m \in \Z^p$ is a SLOPE pattern, if either $m = 0$,
or, if for all $l \in [\|m\|_\infty]$, there exists $j \in [p]$ such
that $|m_j| = l$. We denote the set of all SLOPE patterns of dimension
$p$ by $\mM_p$. Moreover, for $x \in \R^p$, we define $\patt(x) \in
\mM_p$ through the following.

\begin{enumerate}

\item $\sign(\patt(x)) = \sign(x)$

\item $|x_i| = |x_j| \implies |\patt(x)_i| = |\patt(x)_j|$ 

\item $|x_i| > |x_j| \implies |\patt(x)_i| > |\patt(x)_j|$

\end{enumerate}

\end{definition}

\begin{example}

For $x = (3.1,-1.2,0,-3.1)'$, we have $\patt(x) = (2,-1,0,-2)'$. For
$x \in \R^4$ with $\patt(x) = (0,2,1,-2)'$, we have $\sign(x) =
(0,1,1,-1)'$ and $|x_2| = |x_4| > |x_3| > x_1 = 0$. The set of all
SLOPE patterns in $\R^2$ is given by
\begin{multline*}
\mM_2 = \{(0,0)',(1,0)',(-1,0)',(0,1)',(0,-1)',(1,1)',(1,-1)',(-1,1)',(-1,-1)',\\
(2,1)',(-2,1)', (2,-1)',(-2,-1)',(1,2)',(-1,2)',(1,-2)',(-1,-2)'\}.
\end{multline*}

\end{example}

The main geometric object of study in this section is the signed
permutahedron, which constitutes the dual of the sorted-$\ell_1$-norm unit ball
(Proposition~\ref{prop:subdiff-SLOPE} in
Appendix~\ref{app:proof-SLOPE}) and is defined as
$$
P_w^\pm = \conv\left\{(\sigma_1 w_{\pi(1)},\dots,\sigma_p
w_{\pi(p)})' : \sigma_1,\dots,\sigma_p \in \{-1,1\}, \pi \in
\mS_p\right\}.
$$
The shape of this polytope is illustrated in
Figure~\ref{fig:SLOPE-patterns} (in two dimensions) and in
Figure~\ref{fig:SLOPE-patterns3d} (in three dimensions). Also of
importance will be the permutahedron, defined by
$$
P_w = \conv\left\{(w_{\pi(1)},\dots,w_{\pi(p)})' : \pi \in \mS_p\right\}.
$$
The permutahedron is, in fact, a face of the signed permutahedron
$P_w^\pm$. We denote the subdifferential of the sorted-$\ell_1$-norm at $x \in
\R^p$ by $\dnormx{w}(x)$. Any $\dnormx{w}(x)$ is a face of $P_w^\pm$,
which we shall denote by $F_w(x)$ in the following.

SLOPE patterns $m$ having only positive components can be interpreted as
an ordered partition of $[p]$, where the the smallest and largest
element of this partition is the set $\{j : m_j = 1\}$ and the set
$\{j : m_j = \|m\|_\infty\}$, respectively. It is well known that
there is a one-to-one relationship between the elements of an ordered
partition and the faces of the permutahedron \citep[see
e.g.][]{MaesKappen92,Simion97,Ziegler12}. Instigated by this, we show
in Theorem~\ref{thm:SLOPE-pattern} that this result can, indeed, be
extended to a one-to-one relationship between all SLOPE patterns and the
non-empty faces of the signed permutahedron, which we denote by
$\mF_0(P_w^\pm)$.

\begin{theorem} \label{thm:SLOPE-pattern}
The mapping $m \in \mM_p \mapsto F_w(m) = \dnormx{w}(m)$ is a
bijection between the SLOPE patterns $\mM_p$ and $\mF_0(P_w^\pm)$, the
non-empty faces of the signed permutahedron $P_w^\pm$. In addition, the
following holds.

\begin{enumerate}

\item The codimension of $F_w(m)$ is given by $\|m\|_\infty$.

\item We have $F_w(x) = F_w(\patt(x))$.

\end{enumerate}

\end{theorem}

The assumption that components of $w$ are strictly decreasing and
non-zero is important. For example, if $w_1 = \dots = w_p > 0$, the
signed permutahedron is just a cube and clearly, there is no
one-to-one relationship between the set SLOPE patterns and the set of
faces of the cube. A similar situation arises if $w$ contains zero
components. As can be seen when $p = 2$ and $w_2 = 0$, the
sorted-$\ell_1$-norm is the supremum norm and the corresponding dual
unit ball is the unit cross-polytope in $\R^2$, whose faces cannot
bijectively be mapped to $\mM_2$ given in the example above.

\begin{example}
We now describe the faces $F_w(m)$, $m \in \mM_2$, of the signed
permutahedron $P_w^\pm$ when $w = (3.5,1.5)'$. In the following, we
use the fact that -- up to an orthogonal transformation described in
Lemma~\ref{lem:trans-prop} -- $F_w(m)$ is equal to $F_w(\tilde m)$ for
some $\tilde m$, a non-negative and non-increasing SLOPE pattern. The
relationship between the SLOPE patterns $m \in \mM_2$ and faces of the
signed permutahedron $P_w^\pm$ are listed below and illustrated in
Figure~\ref{fig:SLOPE-patterns}. Note that $\codim(F_w(m)) =
\|m\|_\infty$.

\bigskip

\noindent
\begin{tabular}{llcl} 
pattern $\tilde m$ & face $F_w(\tilde m)$ & codim.  
& faces $F_w(m)$ isometric to $F_w(\tilde m)$ \\ \hline \\ [-1.5ex]
$\tilde m = (0,0)'$ & signed permutahedron $P_w^\pm$ & 0 & --\\[0.5ex]
$\tilde m = (1,0)'$ & segment $\{3.5\} \times [-1.5,1.5]$ & 1 & $m \in \{(-1,0)',\pm(0,1)'\}$
 \\[0.5ex]
$\tilde m = (1,1)'$ & permutahedron $P_w$ & 1 & $m \in \{(-1,-1)',\pm(1,-1)'\}$ \\ [0.5ex]
$\tilde m = (2,1)'$ & point: $(3.5,1.5)'$ & 2 &  
$m \in \{(-2,-1)',\pm(2,-1)',\pm(1,2)',\pm(1,-2)'\}$ \\	
\end{tabular}

\begin{figure}[htbp]

\centering 
\includegraphics{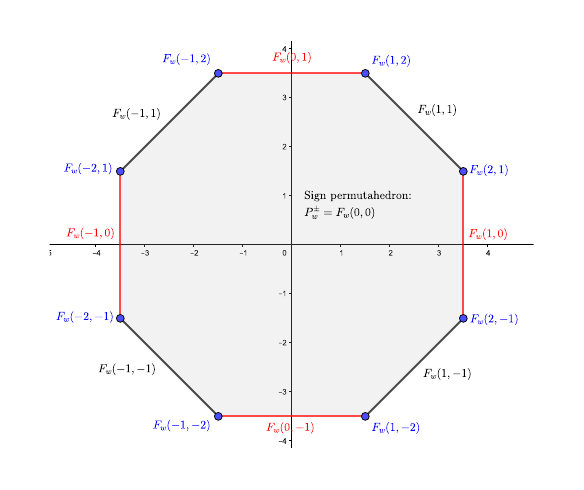}

\caption{\label{fig:SLOPE-patterns} Illustration of the relationship
between the SLOPE patterns and the faces of the signed permutahedron
$P_w^\pm$ for $w=(3.5,1.5)'$ through subdifferential calculus, see
Proposition~\ref{prop:subdiff-face} in Appendix~\ref{app:subdiff-norm}
and Proposition~\ref{prop:subdiff-SLOPE} in
Appendix~\ref{app:proof-SLOPE}. Note that $F_w(m) = \dnormx{w}(m)$.
Faces having the same color are isometric. One may notice that
$\codim(F_w(m)) = \|m\|_\infty$.}

\end{figure}

\end{example}

Analogously to the accessible sign vectors for LASSO and BP, for a
given $X$, we introduce the notion of accessible SLOPE patterns.
 
\begin{definition}[Accessible SLOPE pattern] \label{def:SLOPE-acc}
Let $X \in \R^{n \times p}$ and $m \in \mM_p$. We say that $m$ is an
accessible SLOPE pattern with respect to $X$ if
$$
\exists y \in \R^n \text{ and } \exists \hat\beta \in \Sslope(y)
\text{ such that } \patt(\hat\beta) = m.
$$
\end{definition}

We now provide a geometric and analytic characterization of accessible
SLOPE patterns.

\pagebreak[2]

\begin{theorem}[Characterization of accessible SLOPE patterns] \label{thm:SLOPE-acc}
Let $X \in \R^{n \times p}$. 

\begin{enumerate}

\item Geometric characterization: A SLOPE pattern $m \in \mM_p$ is
accessible with respect to $X$ if and only if $\row(X)$ intersects the
face $F_w(m)$.

\item Analytic characterization: A SLOPE pattern $m \in \mM_p$ is
accessible with respect to $X$ if and only if the implication
$$
Xb = Xm \implies \|b\|_w \geq \|m\|_w
$$
holds.

\end{enumerate}

\end{theorem}

We point out that the analytic characterization allows to check
accessibility of a particular SLOPE pattern by in fact minimizing a
BP-like problem where the $\ell_1$-norm is replaced by the
sorted-$\ell_1$-norm. This in turn can give insight on whether the
corresponding face of the signed permutahedron is intersected by
$\row(X)$.

Also note that the set of accessible SLOPE patterns is invariant by
scaling $w$ with a constant, since $\row(X)$ intersects $F_w(m)$ if
and only if $\row(X)$ intersects $F_{\lambda w}(m)$ with $\lambda >
0$. The following corollary, which is in line with Theorem~2.1 very
recently given in \cite{KremerEtAl22}, is a straightforward
consequence of Theorems~\ref{thm:uni-pen}, \ref{thm:SLOPE-pattern} and
\ref{thm:SLOPE-acc}.
\begin{corollary}
\label{cor:sparsitySLOPE} Let $X \in \R^{n \times p}$. If $\row(X)$
does not intersect any face of $P_w^\pm$ with codimension larger than
$\rk(X)$, then for all $y \in \R^n$, $\hat\beta_w(y)$, the unique
element of $\Sslope(y)$, satisfies
$\|\patt(\hat{\beta}_w(y))\|_\infty \leq \rk(X)$.
\end{corollary}
Corollary~\ref{cor:sparsitySLOPE} generalizes the well-known fact
that, when uniqueness occurs, the LASSO minimizer has less than
$\rk(X)$ non-null components. Indeed, the above corollary shows that
when the SLOPE minimizer is unique, the number of non-null clusters is
less than or equal to $\rk(X)$.

\begin{example}
We illustrate the criterion for accessible SLOPE patterns from
Theorem~\ref{thm:SLOPE-acc} for $w = (5.5,3.5,1.5)'$ and $X$ given by
$$
X = \begin{pmatrix} 8 & 5 & 8 \\ 10 & 1.25 & -6 \end{pmatrix}.
$$
Table~\ref{tab:SLOPE-acc} lists all accessible non-null SLOPE patterns
($m=0$ is always accessible through $y=0$), the geometric illustration
is shown in Figure~\ref{fig:SLOPE-patterns3d}.

\begin{table}[htbp]

\begin{tabular}{l l l l l}
colour & type & intersection $\neq \emptyset$ & face intersected isom. to & SLOPE patt. 
\\ \hline \\[-2ex]
\textcolor{orange}{orange} & segments & $\row(X) \cap F_w(\pm
(1,0,0))$ & $\{5.5\} \times P^\pm_{(3.5,1.5)}$ & $\pm (1,0,0)$ \\
\textcolor{red}{red} & segments & $\row(X) \cap F_w(\pm (1,1,1))$ & $P_{(5.5,3.5,1.5)}$ & 
$\pm (1,1,1)$ \\ 
black & segments & $\row(X) \cap F_w(\pm (0,0,1))$ & $\{5.5\} \times
P^\pm_{(3.5,1.5)}$ & $\pm (0,0,1)$ \\
\textcolor{pink}{pink} & segments & $\row(X) \cap F_w(\pm (-1,0,1))$ &
$P_{(5.5,3.5)} \times [-1.5,1.5]$ & $\pm (-1,0,1)$ \\
\textcolor{purple}{purple} & points  &$\row(X) \cap F_w(\pm (2,0,-1))$
& $\{5.5\} \times \{3.5\} \times [-1.5,1.5]$ & $\pm (2,0,-1)$ \\
\textcolor{green}{green}  & points & $\row(X) \cap F_w(\pm (2,1,1))$ & $\{5.5\} \times
P_{(3.5,1.5)}$ & $\pm(2,1,1)$ \\
\textcolor{blue}{blue} & points & $\row(X) \cap F_w(\pm (1,1,2))$ & $\{5.5\}\times
P_{(3.5,1.5)}$ & $\pm (1,1,2)$ \\
\textcolor{yellow}{yellow} & points & $\row(X) \cap F_w(\pm (-1,0,2))$ & $\{5.5\}\times
\{3.5\} \times [-1.5,1.5]$ & $\pm (-1,0,2)$
\end{tabular}

\medskip

\caption{\label{tab:SLOPE-acc} Accessible SLOPE patterns with respect to 
$X = \left( \protect\begin{smallmatrix} 8 & 5 & 8 \\ 10 & 1.25 & -6 \protect\end{smallmatrix}\right)$ 
and $w = (5.5,3.5,1.5)'$.}

\end{table}

\begin{figure}[htbp] 

\centering

\includegraphics{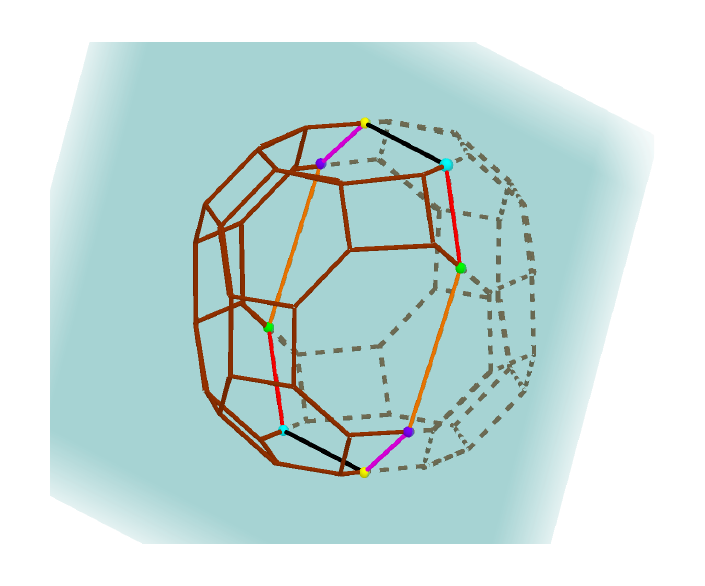}

\caption{\label{fig:SLOPE-patterns3d} Illustration of the signed
permutahedron $P_w^\pm$ (in brown) and the plane $\row(X)$ (in light
blue). Because $\rk(X) = 2$ and $\row(X)$ does not intersect any
vertex of $P_w^\pm$ (the faces with codimension equal to 3), the SLOPE
estimator $\hat\beta_w(y)$ is unique for all values of $y \in \R^2$.
Colored segments and points are the intersections between $\row(X)$
and the faces of $P_w^\pm$, determining the accessible SLOPE patterns
shown in Table~\ref{tab:SLOPE-acc}. For example, $m = (2,1,1)'$ is an
accessible SLOPE pattern, which implies that there exists $y \in \R^2$
for which the SLOPE minimizer $\hat\beta_w(y)$ satisfies
$\hat\beta_w(y)_1 > \hat\beta_w(y)_2 = \hat\beta_w(y)_3 > 0$. In
addition, since $m=(2,1,0)'$ is not an accessible pattern, one cannot pick
$y \in \R^2$ for which the SLOPE minimizer satisfies $\hat\beta_w(y)_1
> \hat\beta_w(y)_2 > \hat\beta_w(y)_3 = 0$.}

\end{figure}

\end{example}


\subsection{The SLOPE null polyhedron and a general result}
\label{subsec:sparse_nullpolytope}

In the previous section, we gave a description of accessible SLOPE
patterns based on the intersection of $\row(X)$ with the signed
permutahedron $P_w^\pm$. In this section, our aim is the following:
Given an accessible pattern $m \in \mM_p$, we want to provide the set of
$y \in \R^n$ for which there exists $\hat\beta \in \Sslope(y)$ with
$\patt(\hat\beta) = m$. In other words, we want to describe the set
$$
A_w(m) = \{y \in \R^n : \exists \hat \beta\in \Sslope(y) \text{ \rm where } \patt(\hat\beta) = m\}.
$$
Note that when the SLOPE minimizer is unique, the sets $A_w(m)$ and
$A_w(\tilde m)$ are disjoint for $m \neq \tilde m$, whereas $A_w(m)
\cap A_w(\tilde m) \neq \emptyset$ might occur in case of
non-uniqueness. Clearly, the null pattern $m = 0$ is accessible. The
corresponding set $A_w(0)$, called the \emph{SLOPE null polyhedron},
given by
$$
A_w(0) = \{y \in \R^n : \|X'y\|_w^* \leq 1\}
$$
by Proposition~\ref{prop:ch-pen}. This is the set of all $y$ such that
$X'y \in P_w^\pm$, which is again a polytope. 
The proposition below shows that the faces of this polytope $N_w(m) =
\{f \in \R^n : X'f \in F_w(m)\}$ for the \emph{accessible} SLOPE
patterns $m$ are the cornerstone to describe the sets $A_w(m)$.

\begin{proposition} \label{prop:SLOPE-As}
Let $X \in \R^{n \times p}$. The SLOPE pattern $m \in \mM_p$ is an
accessible SLOPE pattern if and only if $N_w(m) = \{f \in \R^n : X'f \in
F_w(m)\} \neq \emptyset$. In that case, the set $A_w(m)$ is given by
$$
A_w(m) = \left\{y = f + Xb : f \in N_w(m), \patt(b) = m\right\}.
$$
\end{proposition}

Note that Proposition~\ref{prop:SLOPE-As} yields another
characterization of accessible SLOPE patterns, namely that $m$ is
accessible if and only if $N_w(m)$ is a non-empty face of the SLOPE
null polyhedron. In case of non-uniqueness, different patterns may yield
the same face, so one should be aware that there is no bijection
between the accessible SLOPE patterns and the faces of the SLOPE null
polytope. Also note that if $\rk(X) = n$ and we are given the
intersection between $\row(X)$ and $F_w(m)$ for some accessible SLOPE
pattern $m$, we can write $N_w(m) = (XX')^{-1}X(\row(X) \cap F_w(m))$
since
$$
f \in N_w(m) \iff X'f \in \row(X) \cap F_w(m) \iff f \in (XX')^{-1}X(\row(X) \cap F_w(m)).
$$

\begin{example}

Figure~\ref{fig:SLOPE-patterns3d} illustrates the accessible SLOPE
patterns from Theorem~\ref{thm:SLOPE-acc} for $w = (5.5,3.5,1.5)'$ and
$$
X = \begin{pmatrix} 8 & 5 & 8 \\ 10 & 1.25 & -6 \end{pmatrix}.
$$
Now, for every accessible SLOPE pattern,
Figure~\ref{fig:SLOPE-nullpolytope} below provides the set $A_m = \{y
\in \R^2 : \exists \hat\beta \in \Sslope(y) \text{ \rm where }
\patt(\hat\beta) = m\}$ and the SLOPE null polyhedron.

\begin{figure}[htbp]

\centering
\includegraphics[scale=0.3]{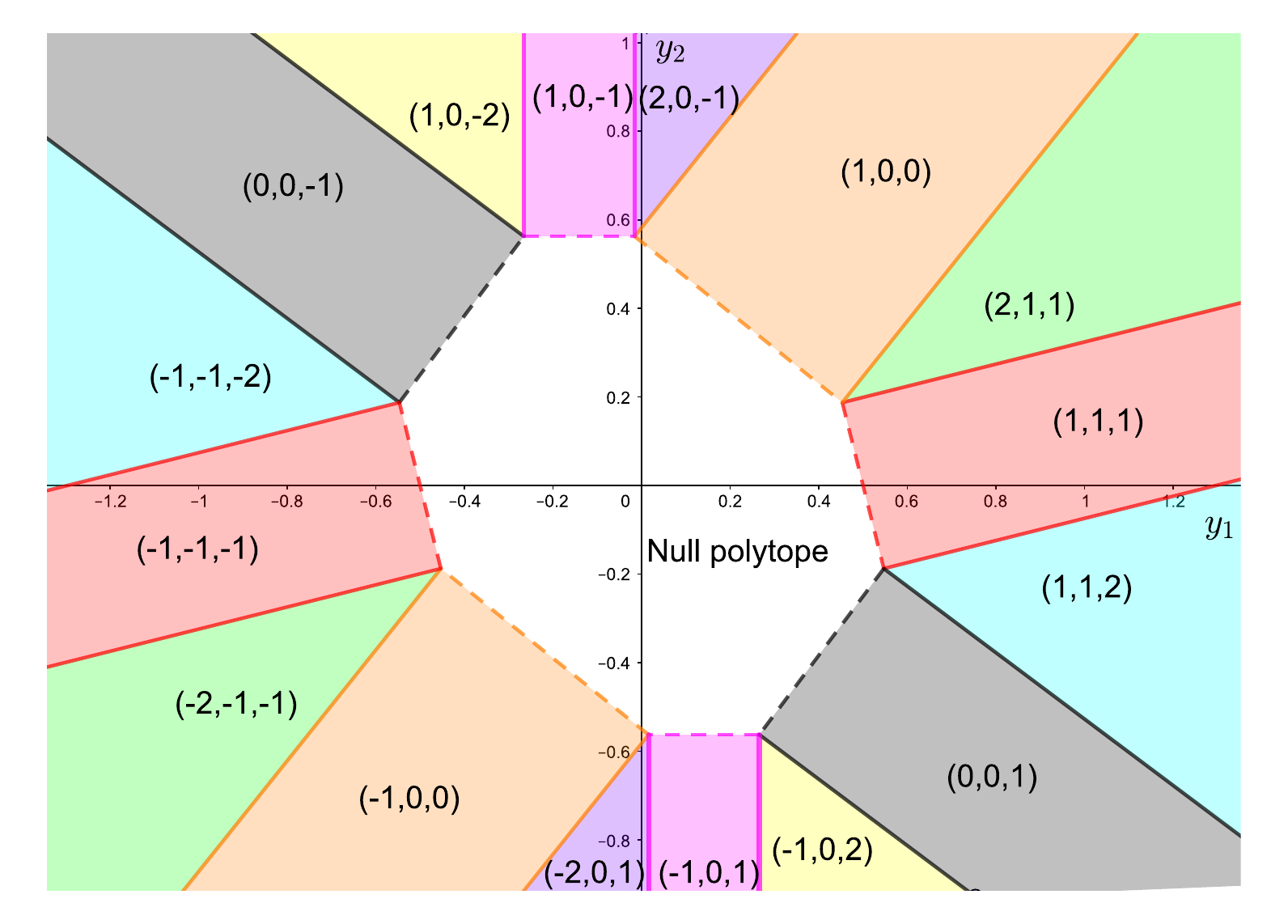}

\caption{\label{fig:SLOPE-nullpolytope} Illustration of the SLOPE null
polytope and the accessible patterns for $X = 
\left(\protect\begin{smallmatrix} 8 & 5 & 8 \\ 10 & 1.25 & -6 \protect\end{smallmatrix}\right)$ 
and $w = (5.5,3.5,1.5)'$. The resulting accessible patterns are $\{\pm
(1,0,0), \pm (1,1,1), \pm (0,0,1), \pm (-1,0,1), \pm (2,0,-1), \pm
(2,1,1), \pm (1,1,2), \pm (-1,0,2)\}$, each associated with a face of
the polytope. Depicted also are the sets $A_w(m) =\{y \in \R^2 : \exists
\hat\beta \in \Sslope(y) \text{ \rm with } \patt(\hat\beta) = m\}$
for each accessible pattern.}

\end{figure}

\end{example}

Note that the SLOPE null polyhedron $A_w(0)$ can also be interpreted
as the set of SLOPE residuals in the sense that $\hat u = y -
X\hat\beta$ is the projection of $y$ onto $A_w(0)$ whenever $\hat\beta
\in \Sslope(y)$ \citep{Minami20}. Or put differently again, we can
decompose $y$ as $y = X\hat\beta + \hat u$, where $X\hat\beta$ is the
SLOPE fit and $\hat u \in A_w(0)$, the set of all values that lead to
a zero SLOPE minimizer. This property is well known also for the
LASSO, \citep[c.f.][]{TibshiraniTaylor12}. In fact, it is
straightforward to see from Proposition~\ref{prop:ch-pen} that the
same considerations hold for all problems as defined in
(\ref{eq:pen}). For completeness, we summarize this in the following
proposition which holds for arbitrary norms.

\begin{proposition} \label{prop:convexnullset}
Let $X \in \R^{n \times p}$ and $y \in \R^n$ and let $\|.\|$ be a norm
on $\R^p$. Define the convex null set $A_\emptyset = \{u \in \R^n :
\|X'u\|^* \leq 1\}$. We then have $\Spen(u) = \{0\}$ for all $u \in
A_\emptyset$, and any $\hat\beta \in \Spen(y)$ satisfies $y =
X\hat\beta + \hat u$ with $\hat u \in A_\emptyset$. Moreover, $\hat u$
is the projection of $y$ onto $A_\emptyset$.
\end{proposition}

\section{Conclusion and perspective} In Theorem \ref{thm:uni-pen}, we
provide a necessary and sufficient condition for the uniqueness of
penalized least-squares estimators whose penalty term is given by a
norm with a polytope-shaped unit ball.  To the best of our knowledge,
this kind of uniqueness has not been treated in this generality
before, only a necessary and sufficient condition in the special case
of the LASSO has been available. Our condition involves a new
geometric approach that allows to also investigate other properties of
these types of methods. A central role in this approach is played by
the unit ball of the norm that is dual to the penalizing term, denoted
by $B^*$. For the LASSO, it is fairly straightforward to see that
every possible signed model corresponds to a face of $B^*$, the unit
cube in this case. For SLOPE, we show that $B^*$ is, in fact, given by
the so-called signed permutahedron. We also show the highly
non-trivial fact that there is a one-to-one correspondence between the
faces of this signed permutahedron and the so-called SLOPE patterns,
which contain the information about zero components, signs, clusters,
and the ordering in a SLOPE solution. Our findings illustrate the
intrinsic connection between the faces of the geometric object $B^*$
and the type of patterns the corresponding penalized method can
uncover. This suggests to further explore this link generally in
penalized estimation, which could, for example, be accomplished by the
implicit definition that patterns are equivalence classes of elements
sharing the same subdifferential with respect to the penalty term.
Another natural direction for extending the uniqueness result (and
also the results for pattern selection) would be to consider even more
general penalties to also encompass methods such as the generalized
LASSO and related procedures which are currently not covered in our
setting.

\section{Acknowledgments} \label{sec:acknowledgements}

We would like to thank Jan Mielniczuk and \'Swiatos\l{}aw Gal for
their insightful comments on the paper. Patrick Tardivel's
affiliation, the Institute of Mathematics in Burgundy (IMB), receives
support from the EIPHI Graduate School (contract ANR-17-EURE-0002).

\appendix

\section{Appendix -- Accessible sign vectors for LASSO and BP} \label{app:LASSO_BP}

We start by introducing the notion of accessible sign vectors for
LASSO and BP problems.

\begin{definition}[Accessible sign vectors for LASSO and BP] 
Let $X \in \R^{n \times p}$, $\sigma \in \{-1,0,1\}^p$, and $\lambda >
0$. We say that $\sigma$ is an accessible sign vector for LASSO (or
BP) with respect to $X$, if there exists $y \in \R^n$ and $\hat\beta \in \SlassoL(y)$ (or
there exists $y \in \col(X)$ and $\hat\beta \in \Sbp(y)$,
respectively), such that $\sign(\hat\beta) = \sigma$.
\end{definition}

The following theorem provides a geometric characterization of
accessible sign vectors for LASSO and BP based on faces of the unit
cube $[-1,1]^p$ and the vector space $\row(X)$. First, note that
sub-differential calculus of the $\ell_1$-norm at $\sigma \in
\{-1,0,1\}^p$ gives
$$
\dnormx{1}(\sigma) = E_1 \times \dots \times E_p \text{ \rm with } 
E_j = \begin{cases} \{\sigma_j\} & |\sigma_j| = 1 \\
[-1,1] & \sigma_j = 0,
\end{cases}
$$
where $\dnormx{1}(x)$ denotes the subdifferential of the $\ell_1$-norm at
$x \in \R^p$, see Appendices~\ref{app:subdiffs} and
\ref{app:subdiff-norm} for more details. Therefore, the mapping
$\sigma \mapsto \dnormx{1}(\sigma)$ is a bijection between sign
vectors in $\{-1,0,1\}^p$ and faces of the unit cube in $\R^p$. We let
$F_1(\sigma) = \dnormx{1}(\sigma)$ in the following. For completeness,
Theorem~\ref{thm:LASSO_BP-acc} also contains an analytic
characterization of accessibility.

\begin{theorem}[Characterization of accessible LASSO and BP sign vectors] 
\label{thm:LASSO_BP-acc} Let $X \in \R^{n \times p}$ and $\lambda >
0$.

\begin{enumerate}

\item Geometric characterization: A sign vector $\sigma \in
\{-1,0,1\}^p$ is accessible for LASSO or BP with
respect to $X$ if and only if $\row(X)$ intersects the face
$F_1(\sigma)$.

\item Analytic characterization: A sign vector $\sigma \in
\{-1,0,1\}^p$ is accessible for LASSO or BP with respect to $X$ if and
only if the implication
$$
Xb = X\sigma \implies \|b\|_1 \geq \|\sigma\|_1
$$
holds.
\end{enumerate}
\end{theorem}

The analytic characterization for accessible sign vectors is, in fact,
closely related to the identifiability condition given in
\cite{TardivelBogdan22}, in which the inequality above is replaced
by a strict inequality. In high-dimensional linear regression, this
condition is necessary and sufficient for sign recovery of thresholded
LASSO and thresholded BP \citep{TardivelBogdan22}, as well as for
so-called thresholded justice pursuit \citep{DesclouxEtAl22}, a
method closely related to BP. We point out that the analytic
characterization allows to check accessibility of a particular sign
vector simply by solving a BP problem, which in turn gives insight on
whether the corresponding face of the unit cube is intersected by
$\row(X)$. In practice, one does not even need an accurate numerical
solver to check whether a sign vector $\sigma \in \{-1,0,1\}^p$ is
accessible, when the BP problem is uniquely solvable: if we are given
an approximate minimizer $\tilde\beta$ for the BP problem with $y =
X\sigma$ that satisfies $\|\tilde\beta - \hat \beta\|_\infty < 1/2$,
where $\hat\beta$ is the exact minimizer, it suffices to check whether
$\sign(\round(\tilde\beta)) = \sigma$, where $\round(.)$ rounds
componentwise to the closest integer. In that case, $\sigma$ is
accessible, whereas $\sigma$ is not accessible if
$\sign(\round(\tilde\beta)) \neq \sigma$, as outlined in
Corollary~\ref{cor:BP-acc} in Appendix~\ref{app:proof-LASSO_BP-acc}.
This approach to check accessibility was used in
\cite{TardivelBogdan22} to derive the so-called identifiability
curve.

Note that Theorem~\ref{thm:LASSO_BP-acc} reveals that whether a sign
vector is accessible for LASSO does not depend on the value of the
tuning parameter $\lambda$. We also point out that
Theorems~\ref{thm:uni-pen} and \ref{thm:LASSO_BP-acc} allow to deduce
that the number of non-null components of the LASSO is always less
than or equal to $\rk(X)$ when the solutions are unique. Indeed, if
the LASSO minimizer is unique, according to Theorem~\ref{thm:uni-pen},
$\row(X)$ does not intersect a face of $[-1,1]^p$ associated to a sign
vector having more than $\rk(X)$ non-null components, i.e., a face
whose codimension is larger than $\rk(X)$. This implies that only sign
vectors with at most $\rk(X)$ components different to zero are
accessible. For the LASSO, this is a refined version of the well-known
fact that, in case the estimator is unique, at most $n$ components can
be non-zero \citep[see e.g.][]{Tibshirani13,OsborneEtAl00}. 

\section{Appendix -- Proofs} \label{app:proofs}

In the appendix, we additionally make use of the following notation.
Let $A$ be a matrix. We use the symbol $A_j$ to denote the $j$-the
column of $A$. For an index set $I$, $A_I$ is the matrix containing
columns with indices in $I$ only. For a vector $x$, $\supp(x)$
contains the indices of the non-zero components of $x$. The symbol
$|x|_{(j)}$ denotes the $j$-th order statistic of the absolute values
of the components of $x$, i.e., $|x|_{(1)} \geq |x|_{(2)} \geq \dots$.
Let $l,k \in \N$ with $l \leq k$, then $[l:k]$ denotes the set
$\{l,l+1,\dots,k\}$. We let $\1_m$ stand for the vector $(1,\dots,1)'
\in \R^m$. All inequalities involving vectors are understood
componentwise.

\subsection{Facts about subdifferentials} \label{app:subdiffs}

We remind the reader of some definitions and facts on subgradients and
subdifferentials. The following can, for instance, be found in
\cite{HiriartLemarechal93}. For a function $f: \R^p \to \R$, a vector
$s \in \R^p$ is a \emph{subgradient of $f$ at $x \in \R^p$} if
\begin{equation} \label{eq:subdiff-def}
f(z) \geq f(x) + s'(z-x) \;\; \forall z \in \R^p.
\end{equation}
The set of all subgradients of $f$ at $x$, which is a convex set, is
called the \emph{subdifferential of $f$ at $x$}, denoted by
$\partial_f(x)$. It is straightforward to characterize the minimizer
of a function in the following way
\begin{equation} \label{eq:subdiff-min}
x^* \in \Argmin f \iff 0 \in \partial_f(x^*).
\end{equation}

While convexity of $f$ is not necessary for the above statement, the
use of subdifferentials is an especially important tool when this is
the case. Given that $f$ is convex, subdifferentiability is also a
local property in the sense that for any $\delta > 0$, we have
\begin{equation} \label{eq:subdiff-local}
s \in \partial_f(x) \iff f(x+h) \geq f(x) + s'h \text{ for all } h : \|h\|_\infty \leq \delta.
\end{equation}

\subsection{Facts about polytopes} \label{app:polytopes}

We report some basic definitions and facts on polytopes, which we will
use throughout the article and, in particular, in the proofs in
subsequent sections. The following can, for instance, be found in the
excellent textbooks by \cite{Gruber07} and \cite{Ziegler12}.


A set $P_\mV \subseteq \R^p$ is called a $\mV$-polytope, if it is the
convex hull of a finite set of points in $\R^p$, namely,
$$
P_\mV = \conv(V_1,\dots,V_k) = \conv(V)
$$
for $V = (V_1 \dots V_k) \in \R^{p \times k}$. A set $P_\mH \subseteq
\R^p$ is called an $\mH$-polyhedron, if it is the intersection of a
finite number of half-spaces, namely,
$$
P_\mH = \bigcap_{l=1}^m \{x \in \R^p : A_l'x \leq b_l\} = \{x \in \R^p : A'x \leq b\},
$$
for some $A = (A_1 \dots A_m) \in \R^{p \times m}$ and $b \in \R^m$. A
bounded $\mH$-polyhedron is called $\mH$-polytope. A set $P \subseteq
\R^p$ is an $\mH$-polytope if and only if it is a $\mV$-polytope. We
therefore simply use the term \emph{polytope} in the following. The
\emph{dimension} $\dim(P)$ of a polytope is given by the dimension of
$\aff(P)$, the affine subspace spanned by $P$, and its codimension by
$\codim(P) = p - \dim(P)$.  A \emph{face} $F$ of $P$ is any subset $F
\subseteq P$ that satisfies
$$
F = \{x \in P : a'x = b_0\}, \text{ where } P \subseteq \{x : a'x \leq b_0\},
$$
for some $a \in \R^p$ and $b_0 \in \R$. Such an inequality $a'x \leq
b_0$ is called a \emph{valid inequality} of $P$. Note that $F =
\emptyset$ and $F = P$ are faces of $P$ and that any face $F$ is again
a polytope. A face $F \neq P$ is called \emph{proper}. A face of
dimension $0$ is called \emph{vertex}, and we denote the set of all
vertices of $P$ by $\vert(P)$. This set satisfies $\vert(P) \subseteq
\{V_1,\dots,V_k\}$, where $P = \conv(V_1,\dots,V_k)$. A point $x_0 \in
P$ lies in $\relint(P)$, the \emph{relative interior} of $P$, if $x_0$
is not contained in a proper face of $P$. Finally, the \emph{(polar)
dual} of $P$ is defined as
$$
P^* = \{s \in \R^p : s'x \leq 1 \, \forall x \in P\},
$$
which is again a polytope. We now list a number of useful facts about
polytopes involving the above definitions, which are used throughout
the article. These properties can either be found explicitly or as a
straightforward consequence of properties listed in the above
mentioned references.

\begin{proposition} \label{prop:polytope-facts}

Let $P \in \R^p$ be a polytope given by $P = \conv(V)$, where $V =
(V_1,\dots,V_k) \in \R^{p \times k}$, and denote by $P^*$ the dual of
$P$. For simplicity, we assume that $\vert(P) = \{V_1,\dots,V_k\}$.
Moreover, let $0 \in P$. The following properties hold.

\begin{enumerate}

\item If $F$ and $\tilde F$ are faces of $P$, then so is $F \cap
\tilde F$.

\item For any face $F$ of $P$, $F = \conv(\vert(P) \cap F)$.


\item Let $D$ be an affine line contained in the affine span of $P$.
If $D \cap \relint(P) \neq \emptyset$ then $D$ intersects a proper
face of $P$.

\item We can write $P^* = \{s \in \R^p : V's \leq \1_k\}$.

\item Any face $F^*$ of $P^*$ can be written as $F^* = \{s \in P^* :
V_I's = \1_{|I|}\}$ for some $I \subseteq [k]$.

\item Let $I \subseteq [k]$. $F = \conv(V_I)$ is a face of $P$ $\iff$
$F^* = \{s \in P^* : V_I's = \1_{|I|}\}$ is a face of $P^*$, where $I$
is the maximal index set in this representation.

In this case, $F^*$ is the dual of $F$ (and vice versa), and $\codim(F^*)
= \rk(V_I)$.

\end{enumerate}

\end{proposition}

\subsection{Facts about subdifferentials of norms with polytope unit
balls} \label{app:subdiff-norm}

We now consider subdifferentials of norms and list several properties
in the following. In particular, we show in
Proposition~\ref{prop:subdiff-norm} that the subdifferential of a norm
evaluated at zero is simply given by the unit ball of the
corresponding dual norm, a fact that will be used throughout
subsequent proofs. Proposition~\ref{prop:subdiff-face} then shows that
all faces of this dual norm unit ball can be represented by a
subdifferential of the original norm, provided that this norm is such
that its unit ball, and therefore also the unit ball of its dual norm,
are given by a polytope. Lemma~\ref{lem:subdiff-face-tech} contains a
technical result needed for the proof of Theorem~\ref{thm:uni-pen}.

A version of the following proposition -- which holds independently of
the shape of the unit ball of the norm under consideration -- can also
be found in \cite{HiriartLemarechal93}.

\begin{proposition} \label{prop:subdiff-norm}
Let $\|.\|$ be a norm on $\R^p$, and let $\|.\|^*$ denote the dual
norm. Then the following holds.

\begin{enumerate}

\item The subdifferential of $\|.\|$ at $0$ is given by
$$
\dnorm(0) = \{s \in \R^p : \|s\|^* \leq 1\}.
$$

\item In general, the subdifferential of $\|.\|$ at $x$ is given by
$$
\dnorm(x) = \{s \in \R^p : \|s\|^* \leq 1, s'x = \|x\|\}.
$$

\end{enumerate}

\end{proposition}

\begin{proof}

It suffices to show {\emph{2})}. By definition, we have
$$
\dnorm(x) = \{s \in \R^p: \|v\| \geq \|x\| + s'(v-x) \; \forall v \in \R^p\}
$$
Take $s \in \dnorm(x)$. When $v = 0$, we get $s'x \geq \|x\|$. When $v
= 2x$, we may deduce that $s'x \leq \|x\|$, implying that $s'x =
\|x\|$ must hold. This also implies $\|v\| \geq s'v$ for all $v \in
\R^p$, so that $s \in B^*$, yielding
$$
\dnorm(x) \subseteq \{s \in B^*: s'x = \|x\|\}.
$$
To see that also the converse is true, take any $s \in B^*$ satisfying
$s'x = \|x\|$. Now, take any $v \in \R^p$. Clearly $\|v\| \geq s'v =
\|x\| + s'(v-x)$, implying that
$$
\{s \in B^*: s'x = \|x\|\} \subseteq \dnorm(x).
$$
\end{proof}

\begin{proposition} \label{prop:subdiff-face}
Let $\|.\|$ be a norm whose unit ball $B$ is the polytope $\conv(V)$
for some $V = (V_1 \dots V_k) \in \R^{p \times k}$. Let $F \subseteq
B^*$, where $B^*$ is the dual norm unit ball, with $F \neq \emptyset$.
Then
$$
F \text{ is a face of } B^* \iff F = \dnorm(x) \text{ for some } x \in \R^p.
$$
\end{proposition}

\begin{proof}
($\implies$) If $F = B^*$, then $x = 0$ by
Proposition~\ref{prop:subdiff-norm}. If $F$ is a proper face, we can
write $F = \{s \in B^* : V_I's = \1_{|I|}\}$ for some $I \subseteq
[k]$, where $I$ is the maximal set satisfying this. Let $x = \sum_{l
\in I} V_l$. Since $x/|I| \in \conv(V_I)$, a proper and non-empty face
of $B$, we have $\|x\| = |I|$. Note that for $s \in B^*$, we have
$s'V_l \leq 1$, so that
$$
s \in \dnorm(x) \iff s'x = \sum_{l \in I} V_l's = \|x\| = |I|
\iff V_l's = 1 \;\; \forall l \in I \iff s \in F.
$$ 
($\impliedby$) If $F = \dnorm(x)$, then $F = \{s \in B^* : s'x =
\|x\|\}$ by Proposition~\ref{prop:subdiff-norm}. Since $(x/\|x\|)'s
\leq 1$ clearly is a valid inequality for all $s \in B^*$, $F$ is a
face of $B^*$.
\end{proof}

\begin{lemma} \label{lem:subdiff-face-tech}
Let $\|.\|$ be a norm whose unit ball $B$ is the polytope $\conv(V)$
for some $V = (V_1 \dots V_k) \in \R^{p \times k}$. Let $F = \{s \in
B^* : V_I's = \1_{|I|}\}$ be a face of $B^*$, the dual norm unit ball,
and let $I$ be the maximal set satisfying this. Then the following
holds.
$$
F \subseteq \dnorm(b) \implies b \in \col(V_I).
$$

\end{lemma}

\begin{proof}
Since $b/\|b\| \in B = \conv(V)$, we can write $b = \sum_{l = 1}^k
\alpha_l V_l$ with $\alpha_l \geq 0$ and
$\sum_{l=1}^{k}\alpha_l=\|b\|$. Since $\dnorm(b) = \{s \in B^*: s'b =
\|b\|\}$ and $s'V_l \leq 1$, we have for $A = \supp(\alpha)$ and any
$s \in \dnorm(b)$
$$
\|b\| = s'b = \sum_{l \in A} \alpha_l s'V_l \leq \sum_{l \in A}
\alpha_l = \|b\|.
$$
This implies that $s'V_l = 1$ for all $l \in \supp(\alpha)$, which,
since $F \subseteq \dnorm(b)$, yields $\supp(\alpha) \subseteq I$.
\end{proof}

\subsection{Proofs of Theorems \ref{thm:uni-pen} and \ref{thm:uni-BP}} 
\label{app:proof-1and2}

The proofs of Theorems~\ref{thm:uni-pen} and \ref{thm:uni-BP} follow a
similar outline, with the proof of Theorem~\ref{thm:uni-BP} being more
accessible. We therefore start with the latter one.

\subsubsection{Characterization of BP minimizers and proof of
Theorem~\ref{thm:uni-BP}}

The following characterization of BP minimizers will prove useful in
the following. It can be found in \cite{ZhangEtAl15} and
\cite{Gilbert17}, as well as in general form in \cite{MousaviShen19}.

Let $y \in {\rm col}(X)$ and let $\hat\beta$ satisfy $X\hat\beta = y$
then, $\hat\beta \in \Sbp(y)$ if and only if

\begin{equation} \label{eq:chBP}
\exists z \in \R^n \text{ such that }
\begin{cases} \|X'z\|_\infty \leq 1,\\
X_j'z= \sign(\hat\beta_j) \;\; \forall j \in \supp(\hat\beta).
\end{cases}
\end{equation}

\begin{proof}[Proof of Theorem~\ref{thm:uni-BP}] \leavevmode

($\impliedby$) Let us assume that $\row(X)$ intersects a face $F$ of
$[-1,1]^p$ whose codimension is larger than $\rk(X)$. We show that one
can find some $y \in \col(X)$ for which $\Sbp(y)$ is not a singleton.

The face $F$ can be written as $F = E_1 \times \dots \times E_p$,
where $E_j \in \{\{-1\},\{1\},[-1,1]\}$ for $j \in [p]$. Now, let $J =
\{j \in [p] : |E_j|=1\}$, the set of indices of sets $E_j$ that are
singletons. We have $\codim(F) = |J|$ and, by assumption, $|J| >
\rk(X)$. Now define $\hat\beta \in \R^p$ by setting
$$
\hat\beta_j = 
\begin{cases}
1 & E_j = \{1\} \\
-1 & E_j = \{-1\} \\
0 & j \notin J .
\end{cases}
$$
Clearly, $\supp(\hat\beta) = J$. Set $y = X\hat\beta$. Since $\row(X)$
intersects $F$, there exists $z \in \R^n$ such that $X'z \in F$. This
implies that $\|X'z\|_\infty \leq 1$ and $X_j'z = \hat\beta_j =
\sign(\hat\beta_j)$ for any $j \in \supp(\hat\beta) = J$. Therefore,
by ($\ref{eq:chBP}$), $\hat\beta \in \Sbp(y)$.

To show that $\hat\beta$ is not a unique minimizer, we provide
$\tilde\beta \in \R^p$ with $\tilde\beta \neq \hat\beta$,
$X\tilde\beta = y$ and $\|\tilde\beta\|_1 = \|\hat\beta\|_1$. Since
$|J| > \rk(X)$, the columns of $X_J$ are linearly dependent, so that
we can pick $h \in \ker(X)$, $h \neq 0$ such that $\supp(h) \subseteq
J$ and $\|h\|_\infty < 1$. Since $\|h\|_\infty < 1$, $\sign(\hat\beta
+ h) = \sign(\hat\beta)=\hat\beta$. Let $\tilde\beta = \hat\beta + h$.
Note that $X\tilde\beta = X\hat\beta = y$ and that
\begin{align*}
\|\tilde\beta\|_1 & = \sum_{j=1}^p \sign(\hat\beta_j + h_j)(\hat\beta_j + h_j) = 
\sum_{j=1}^p \sign(\hat\beta_j)\hat\beta_j + \sum_{j \in J} \hat\beta_j h_j 
= \|\hat\beta\|_1 + \sum_{j \in J} (X'z)_j h_j \\ 
& = \|\hat\beta\|_1 + z'Xh = \|\hat\beta\|_1,
\end{align*}
implying that $\tilde\beta \in \Sbp(y)$ also.


($\implies$) We assume that $\hat\beta, \tilde\beta \in \Sbp(y)$ with
$\hat\beta \neq \tilde\beta$ for some $y \in \col(X)$. We need to show
that there exists a face $F$ of $[-1,1]^p$ with $F \cap \row(X) \neq
\emptyset$ and $\codim(F) >  \rk(X)$. Consider $F = E_1 \times \dots
\times E_p$ and $\tilde F = \tilde E_1 \times \dots \times \tilde E_p$
with
$$
E_j = \begin{cases} \{\sign(\hat\beta_j)\} & \text{ if } j \in \supp(\hat\beta) \\
[-1,1] & \text{ if } j \notin \supp(\hat\beta) \end{cases} 
\;\; \text{ and } \;\;
\tilde E_j = \begin{cases} \{\sign(\tilde\beta_j)\} & \text{ if } j \in \supp(\tilde\beta) \\
[-1,1] & \text{ if } j \notin \supp(\tilde\beta). \end{cases} 
$$
Note that for any two minimizers $\hat\beta$ and $\tilde\beta$, we
have $\hat\beta_j\tilde\beta_j \geq 0$ for all $j \in [p]$, since
otherwise $\check\beta = (\hat\beta + \tilde\beta)/2$ satisfies
$X\check\beta = X\hat\beta = X\tilde\beta$ as well as
$\|\check\beta\|_1 < \|\hat\beta\|_1 = \|\tilde\beta\|_1$, which would
lead to a contradiction. We therefore have $\supp(\check\beta) =
\supp(\hat\beta) \cup \supp(\tilde\beta)$. Note that by a convexity
argument, $\check\beta \in \Sbp(y)$ also, so that by (\ref{eq:chBP}),
there exists $\check z \in \R^n$ with $\|X'\check z\|_\infty \leq 1$
and $X_j'\check z = \sign(\check\beta_j)$ for all $j \in
\supp(\check\beta)$. Moreover, $X'\check z \in F \cap \tilde F$ holds.
Now, let $F_0$ be a face of the face $F \cap \tilde F$ of smallest
dimension that still intersects $\row(X)$. We write $F_0 = E_{0,1}
\times \dots \times E_{0,p}$ and let $J_0 = \{j \in [p] : |E_{0,j}| =
1\}$. Note that $\row(X)$ must intersect $F_0$ in its relative
interior $\relint(F_0)$ where
$$
\relint(F_0) = \relint(E_{0,1}) \times \dots \times \relint(E_{0,p}) \;\; \text{ where } 
\relint(E_{0,j}) = \begin{cases} E_{0,j} & j \in J_0 \\ (-1,1) & j \notin J_0, \end{cases}
$$
since otherwise $\row(X)$ intersects a proper face of $F_0$, which
contradicts the assumption that $F_0$ is of minimal dimension. We now
need to show that $\codim(F_0) = |J_0| > \rk(X)$. Assume that $|J_0|
\leq \rk(X)$. The columns of $X_{J_0}$ are linearly dependent since
$X_{J_0}\hat\beta_{J_0} = X\hat\beta = X\tilde\beta = X_{J_0}\tilde
\beta_{J_0}$ with $\hat\beta_{J_0} \neq \tilde\beta_{J_0}$, since both
$\supp(\hat\beta)$ and $\supp(\tilde\beta)$ are subsets of
$\supp(\check \beta) \subseteq J_0$. We therefore have
$$
\dim(\col(X_{J_0})) < |J_0| \leq \rk(X) =  \dim(\col(X))  \;\; \text{ and } \;\;
\col(X)^\perp \subsetneqq \col(X_{J_0})^\perp.
$$
This implies that we can pick $u \in \col(X_{J_0})^\perp \setminus
\col(X)^\perp$ so that $X_{J_0}'u = 0$, but $X'u \neq 0$. Pick $z_0
\in \R^n$ with $X'z_0 \in \relint(F_0)$. The affine line $\{X'(z_0 +
tu) : t \in \R\} \subseteq \row(X)$ intersects the relative interior
$\relint(F_0)$ and is included in the affine span of $F_0$ by
construction of $u$. Therefore, by
Proposition~\ref{prop:polytope-facts}, $\row(X)$ intersects a proper
face of $F_0$, yielding a contradiction.
\end{proof}

\subsubsection{Characterization of penalized minimizers and proof of
Theorem~\ref{thm:uni-pen}}

In the particular and well-studied case in which the norm of the
penalized problem is the $\ell_1$-norm, the solutions to the
corresponding optimization problem can be characterized by the
Karush-Kuhn-Tucker (KKT) conditions for the LASSO, which can be
summarized as follows, see for instance, \cite{BuehlmannVdGeer11}.
\begin{eqnarray}
\hat\beta \in \SlassoL(y) & \iff &  \label{eq:ch-lasso}
\|X'(y-X\hat\beta)\|_\infty \leq \lambda \text{ and } X_j'(y - X\hat\beta) =
\lambda\sign(\hat\beta_j) \; \forall j \in \supp(\hat\beta) \\ \nonumber &
\iff & \|X'(y - X\hat\beta)\|_\infty \leq \lambda \text{ and } \hat\beta'X'(y
- X\hat\beta) = \lambda\|\hat\beta\|_1
\end{eqnarray}
In the above, the supremum norm is the dual to the $\ell_1$-norm. We can
generalize the above characterization for solutions to the penalized
problem from (\ref{eq:pen}) in the following proposition. Note that in
our notation, the tuning parameter $\lambda$ is part of the norm
$\|.\|$.

\begin{proposition} \label{prop:ch-pen}
Let $X \in \R^{n \times p}$, $y \in \R^n$. We have $\hat\beta \in
\Spen(y)$ if and only if
$$
\|X'(y - X\hat\beta)\|^* \leq 1 \text{ and } \hat\beta'X'(y - X\hat\beta) 
= \|\hat\beta\|.
$$
\end{proposition}

\begin{proof}
Using subdifferential calculus, the proof a straightforward
consequence of (\ref{eq:subdiff-min}) and
Proposition~\ref{prop:subdiff-norm}.
\begin{align*}
\hat\beta \in \Spen(y) 
& \iff 0 \in X'(X\hat\beta - y) + \dnorm(\hat\beta) 
\iff X'(y - X\hat\beta) \in \dnorm(\hat\beta) \\ 
& \iff \|X'(y - X\hat\beta)\|^* \leq 1
\text{ and } \hat\beta'X'(y - X\hat\beta) = \|\hat\beta\|.
\end{align*}
\end{proof}

Before finally showing Theorem~\ref{thm:uni-pen}, the following lemma
states that the fitted values are unique over all solutions of the
penalized problem for a given $y$. It is a generalization of Lemma~1
in \cite{Tibshirani13}, who proves this fact for the special case of
the LASSO.

\begin{lemma} \label{lem:fitted-values}
Let $X \in \R^{n \times p}$, $y \in \R^n$. Then $X\hat\beta =
X\tilde\beta$ for all $\hat\beta, \tilde\beta \in \Spen(y)$.
\end{lemma}

\begin{proof}
Assume that $X\hat\beta \neq X\tilde\beta$ for some $\hat\beta,
\tilde\beta \in \Spen(y)$ and let $\check\beta = (\hat\beta +
\tilde\beta)/2$. Because the function $\mu \in \R^n \mapsto
\|y-\mu\|_2^2$ is strictly convex, one may deduce that
$$
\|y - X\check\beta\|_2^2 < 
\frac{1}{2}\|y - X\hat\beta\|_2^2 + \frac{1}{2}\|y - X\tilde{\beta}\|_2^2.
$$ 
Consequently,
$$
\frac{1}{2}\|y - X\check\beta\|_2^2 + \|\check\beta\| <
\frac{1}{2}\left(\frac{1}{2}\|y - X\beta\|_2^2 + \|\beta\| +
\frac{1}{2}\|y - X\tilde\beta\|_2^2 + \|\tilde\beta\|\right),
$$
which contradicts both $\beta$ and $\tilde \beta$ being minimizers.
\end{proof}

\begin{proof}[Proof of Theorem~\ref{thm:uni-pen}] \leavevmode

Throughout the proof, let $B = \conv(V)$ with $V = (V_1 \dots V_k) \in
\R^{p \times k}$.


($\impliedby$) Assume that there exists a face $F$ of $B^*$ that
intersects $\row(X)$ (so that $F$ is non-empty) and satisfies
$\codim(F) > \rk(X)$ (so that $F$ is proper). This implies that there
exists $I \subseteq [k]$ such that
$$
F  = \{s \in B^* : V_I's = \1_{|I|}\},
$$
where $I$ is the maximal index set satisfying this relationship.
Moreover, this implies that  $\conv(V_I)$ is a proper, non-empty face
of $B$ and that we have $\|s\|^* = 1$ for all $s \in F$ and $\|v\|=1$
for all $v \in \conv(V_I)$. We show that non-unique solutions exist.
Define $\hat\beta = \sum_{l \in I} V_l$ and observe that $\|\hat
\beta\| = |I| \| \sum_{l \in I} V_l/|I| \| = |I|$. Pick $z \in \R^n$
with $X'z \in F$, which exists by assumption, and set $y = X\hat\beta
+ z$. Then $\hat\beta \in \Spen(y)$ by Proposition~\ref{prop:ch-pen},
since
$$
\|X'(y - X\hat\beta)\|^* = \|X'z\|^* = 1 \;\; \text{ and
} \;\; \hat\beta'(X'(y-X\hat\beta)) = \hat\beta'X'z = 
\sum_{l \in I} V_l'X'z = |I| = \|\hat\beta\|.
$$
We now construct $\tilde\beta \in \Spen(y)$ with $\tilde\beta \neq
\hat\beta$. Since $\codim(F_I) = \dim(\col(V_I)) > \rk(X)$, we can
pick $h \in \col(V_I) \cap \ker(X)$ with $h \neq 0$. Scale $h$ such
that for $h = \sum_{l \in I} c_l V_l$, we have $\max_{l \in I} |c_l| <
1$, and define $\tilde\beta = \hat \beta + h \neq \hat\beta$. Clearly,
we have $X\tilde\beta = X\hat\beta$. Note that $1+c_l > 0$ and let
$\gamma = \sum_{l \in I} (1+c_l) > 0$. We also have
$$
\|\tilde\beta\| = \gamma \, \left\|\sum_{l \in I} \frac{1 + c_l}{\gamma} V_l\right\| 
= \gamma = \sum_{l \in I}(1 + c_l) 
= |I| + \sum_{l \in I} c_l (X'z)'V_l 
= |I| + (X'z)'h = |I| = \|\hat\beta\|,
$$
proving that $\tilde\beta \in \Spen(y)$ also.


($\implies$) Let us assume that there exists $y \in \R^n$ and
$\hat\beta, \tilde\beta \in \Spen(y)$ with $\hat\beta \neq
\tilde\beta$. We then have
$$
X'(y-X\hat\beta) \in \dnorm(\hat\beta) \;\; \text{ and } \;\; 
X'(y-X\tilde\beta) \in \dnorm(\tilde{\beta}).
$$
Because $X\hat\beta = X\tilde\beta$ by Lemma~\ref{lem:fitted-values},
one may deduce that $\row(X)$ intersects the face $\dnorm(\hat\beta)
\cap \dnorm(\tilde\beta)$. Now, let $F^*$ be a face of $\dnorm(\hat\beta)
\cap \dnorm(\tilde\beta)$ of smallest dimension that intersects
$\row(X)$ and write
$$
F^* = \{s \in B^* : V_I's = \1_{|I|}\},
$$
where $I$ is the largest index set $I \subseteq [k]$ satisfying this
relationship. If $\codim(F^*) = \dim(\col(V_I)) \leq \rk(X)$, consider
the following. Note that we can pick $u \in \R^n$ for which $X'u \neq
0$ and $X'u \in \col(V_I)^\perp$. For this, let $I_0 \subseteq I$ be
such that the columns of $V_{I_0}$ are linearly independent, and
$\col(V_{I_0}) = \col(V_I)$. By Lemma~\ref{lem:subdiff-face-tech}, we
have $\hat\beta, \tilde\beta \in \col(V_{I_0})$, so that we get
$$
XV_{I_0}\gamma = X\hat\beta = X\tilde\beta = XV_{I_0}\tilde\gamma
$$
with $\gamma \neq \tilde\gamma$, implying that the columns of
$XV_{I_0}$ are linearly dependent. But this means that
$$
\rk(XV_I) = \dim(\col(XV_I)) = \dim(\col(XV_{I_0})) < |I_0| = 
\dim(\col(V_{I_0})) =  \dim(\col(V_I)) \leq \rk(X).
$$
Therefore, $\col(XV_I) \subsetneqq \col(X)$ and, consequently,
$\col(X)^\perp \subsetneqq \col(XV_I)^\perp$, so that we can pick $u
\in \col(XV_I)^\perp\setminus\col(X)^\perp$ for which $X'u \neq 0$ and
$X'u \in \col(V_I)^\perp$. Also note that $X'z \in F^*$ for some $z
\in \R^n$ and that $X'z$ lies in the relative interior $\relint(F^*)$,
as otherwise, $\row(X)$ would intersect a face of $\dnorm(\hat\beta)
\cap \dnorm(\tilde\beta)$ of smaller dimension. The affine line
$\{X'(z + tu) : t\in \R\} \subseteq \row(X)$ intersects $\relint(F^*)$
and is included in the affine span of $F^*$ by construction.
Therefore, by Proposition~\ref{prop:polytope-facts}, $\row(X)$
intersects a proper face of $F^*$, yielding a contradiction.
\end{proof}

\subsection{Proof of Proposition~\ref{prop:cond-weak}}

We turn to proving Proposition~\ref{prop:cond-weak}. Note that a set
is negligible with respect to the Lebesgue measure on $\R^{n \times
p}$ if and only if it is negligible with respect to the standard
Gaussian measure on $\R^{n \times p}$. Therefore, to establish
Proposition~\ref{prop:cond-weak}, it suffices to prove the equality
\begin{equation} \label{eq:weak-inequality}
\P_Z \left(\exists y \in \R^n, |\SpenZ(y)| > 1 \right)= 0, 
\text{ \rm where $Z\in \R^{n \times p}$ has 
iid $\mN(0,1)$ entries.}
\end{equation}

Note that $\rk(Z) = \min\{n,p\}$ almost surely. Therefore, when $n
\geq p$, $\ker(Z) = 0$ almost surely and $\SpenZ(y)$ is a singleton
almost surely. We use the following lemma to establish
(\ref{eq:weak-inequality}), where $\N$ stands for the (positive)
natural numbers.

\begin{lemma} \label{lem:almost_surely}
Let $n \in \N$, $q \geq n+1$, and $v \in \R^q$ where $v \neq 0$ is a
fixed vector. If $Z = (Z_1,\dots,Z_n) \in \R^{q \times n}$ has iid
$\mN(0,1)$ entries, then $\P_Z(v \in \col(Z)) = 0$.
\end{lemma}

\begin{proof}
We first prove the result for $q = n+1$. If $v \in \col(Z)$ then
$$
\det(Z_1,\dots,Z_n,v) = 0 \iff \det(Z_1/\|Z_1\|_2,\dots,Z_n/\|Z_n\|_2,v/\|v\|_2) = 0.
$$
Now, because the columns $Z_1/\|Z_1\|_2,\dots,Z_n/\|Z_n\|_2$ follow a
uniform distribution on the $\ell_2$-unit sphere, we can deduce that the
distribution of the random variable
$\det(Z_1/\|Z_1\|_2,\dots,$ $Z_n/\|Z_n\|_2, v/\|v\|_2)$ is equal to the
distribution of
$\det(Z_1/\|Z_1\|_2,\dots,Z_n/\|Z_n\|_2,\zeta/\|\zeta\|_2)$. Here,
$\zeta$ follows a $\N(0, \I_{n+1})$ distribution, independent from
$Z_1,\dots,Z_n$ as conditioning on $\zeta = v$ does not change the
distribution. Finally, the random variable
$$
\det(Z_1/\|Z_1\|_2,\dots,Z_n/\|Z_n\|_2,\zeta/\|\zeta\|_2) =
\frac{1}{\|Z_1\|_2 \times \dots \times \|Z_n\|_2\times \|\zeta\|_2}
\det(Z_1,\dots,Z_{n},\zeta)
$$
is non-zero almost surely. This implies $\P_Z(v \in \col(Z)) = 0$.
When $q > n+1$, let $I \subseteq [q]$ with $|I| = n + 1$ and $v_I \neq
0$. Consequently, $v_I \in \col(\tilde Z)$, where $\tilde Z \in
\R^{(n+1)\times n}$ is obtained by keeping the rows of $Z$ with
indices in $I$. Therefore, $P_Z(v \in \col(Z)) \leq P_{\tilde Z}(v_I
\in \col(\tilde Z)) = 0$, which concludes the proof.
\end{proof}

\begin{proof}[Proof of Proposition \ref{prop:cond-weak}] 
If $p \leq n$, we are done. If $p > n$, let $F_0$ be a proper face of
$B^*$ such that $\codim(F_0) = q > n$. Note that $0 \notin \aff(F_0)$,
the affine space spanned by $F_0$. There exists $A \in \R^{q \times
p}$ with orthonormal rows and $v \in \R^q$, $v \neq 0$ such that
$\aff(F_0) = \{x \in \R^p: Ax=v\}$. Since $AA' = \I_p$, $AZ' \in
\R^{q\times n}$ has iid $\mN(0,1)$ entries. Thus, by
Lemma~\ref{lem:almost_surely}, we have
\begin{equation} \label{eq:negligible}
\P_Z \left(\row(Z) \cap F_0 \neq \emptyset \right) \leq 
\P_Z(\row(Z) \cap \aff(F_0) \neq \emptyset) = \P_Z(v \in \col(AZ')) = 0.  
\end{equation}
According to Theorem~\ref{thm:uni-pen} and since $\rk(Z) = n$ almost
surely, the following equalities hold.
\begin{eqnarray*}
\P_Z\left(\exists y \in \R^n, |\SpenZ(y)| > 1 \right) &=&
\P_Z\left(\bigcup_{\underset{\codim(F) > \rk(Z)}{F \in \mF(P)}}\{\row(Z) \cap F \neq \emptyset\}\right)\\
&=& \P_Z\left(\bigcup_{\underset{\codim(F) > n}{F \in \mF(P)}}\{\row(Z) \cap F\neq \emptyset\}\right) 
= 0.
\end{eqnarray*}
The last equality is a consequence of (\ref{eq:negligible}).
\end{proof}

\subsection{Proof of Theorem~\ref{thm:SLOPE-pattern}} \label{app:proof-SLOPE}

Theorem~\ref{thm:SLOPE-pattern} states that there is a bijection between
the SLOPE patterns and the faces of the signed permutahedron. The basis
for proving this is the fact that the signed permutahedron is the dual
of the sorted-$\ell_1$-norm unit ball, and that any face of it is given by a
subdifferential of the sorted-$\ell_1$-norm by
Proposition~\ref{prop:subdiff-face}.

We start by proving the following proposition which shows that the
subdifferential of the sorted-$\ell_1$-norm at zero is, indeed, the
signed permutahedron, and also characterizes the subdifferential of the
sorted-$\ell_1$-norm for certain values of $x$.

\begin{proposition}  
\label{prop:subdiff-SLOPE} The subdifferential $F_w(x) =
\dnormx{w}(x)$ of the sorted-$\ell_1$-norm exhibits the following properties.
\begin{enumerate}

\item We have 
$F_w(0) = P_w^\pm$.

\item For any $x \in \R^p$ with $x_1 = \dots = x_p > 0$, we have
$F_w(x) = P_w$.

\item For any $x \in \R^p$ with $x_1\ge \dots\ge x_k > x_{k+1}\ge
\dots\ge x_p \ge 0$, we have
$$
F_w(x) = F_{w_{[k]}}(x_{[k]}) \times F_{w_{[k+1:p]}}(x_{[k+1:p]}).
$$

\item Let $0 < k_1 < \dots < k_l < p$ be an arbitrary subdivision of
$[0:p]$, then for any $x \in \R^p$ with $x_1 = \dots = x_{k_1} >
x_{k_1 + 1} = \dots = x_{k_2} > \dots > x_{k_l+1} = \dots = x_p \geq 0$,
we have $\codim\left(F_w(\patt(x))\right) =
\|\patt(x)\|_\infty$ and
$$
F_w(x) = F_w(\patt(x)) =
\begin{cases} 
P_{w_{[k_1]}} \times \dots \times P_{w_{[k_{l-1}+1:k_l]}} \times P_{w_{[k_l+1:p]}} & 
\text{if } x_p > 0 \\
P_{w_{[k_1]}} \times \dots \times P_{w_{[k_{l-1}+1:k_l]}} \times P^\pm_{w_{[k_l+1:p]}} &
\text{if } x_p = 0.
\end{cases}
$$

\end{enumerate}

\end{proposition}

\begin{proof} \emph{1)}
By Proposition~\ref{prop:subdiff-norm}, we may show that $P_w^\pm = B^*$.


($\subseteq$) Take any vertex $W = (\sigma_1 w_{\pi(1)},\dots,\sigma_p
w_{\pi(p)})'$ of $P_w^\pm$ and any $x \in \R^p$ with $\|x\|_w \leq 1$.
We have 
$$
W'x = \sum_{j=1}^p \sigma_j w_{\pi(j)} x_j \leq \sum_{j=1}^p |x_j|
w_{\pi(j)} \leq \sum_{j=1}^p w_j |x|_{(j)} = \|x\|_w \leq 1
$$
and therefore $W \in B^*$. By convexity, $P_w^\pm \subseteq B^*$
follows.


($\supseteq$) Let $a'x \leq b_0$ for some $a \in \R^p$ and $b_0 \in
\R$ be a valid inequality of $P_w^\pm$. We show that this is a valid
inequality of $B^*$ also: Let $W$ be the vertex of $P_w^\pm$ defined
by $W_j = \sign(a_j)w_{\pi^{-1}(j)}$, where the permutation $\pi$
satisfies $|a_{\pi(1)}| \geq \dots \geq |a_{\pi(p)}|$. For any $s \in
B^*$, we have
$$
a's \leq \|a\|_w = \sum_{j=1}^p |a_{\pi(j)}| w_j = 
\sum_{j=1}^p \sign(a_j) a_j w_{\pi^{-1}(j)} = a'W \leq b_0.
$$
Since $P_w^\pm$ can be written as the (finite) intersection of
half-spaces, $P_w^\pm \supseteq B^*$ follows.


\emph{2)} According to Proposition~\ref{prop:subdiff-norm} and 1), we have
$$
F_w(x) = \left\{s \in P_w^\pm : \sum_{j=1}^p s_j = \sum_{j=1}^p w_j\right\}.
$$ 
A vertex $W = (\sigma_1 w_{\pi(1)},\dots,\sigma_p w_{\pi(p)})'$ of
$P_w^\pm$ with $\sigma \in \{-1,1\}^p$ and $\pi \in \mS_p$ then
fulfills $W \in F_w(x)$ if and only if $\sigma_1 = \dots =
\sigma_p=1$. Convexity then yields $F_w(x) = P_w$.


\emph{3)} ($\subseteq$) Let $s \in F_w(x)$. We show that $s_{[k]} \in
F_{w_{[k]}}(x_{[k]})$ and $s_{[k+1:p]} \in
F_{w_{[k+1:p]}}(x_{[k+1:p]})$. Let $e = \frac{x_k - x_{k+1}}{2} > 0$
and $h \in \R^p$ with $\|h\|_\infty < e$. Since the $k$ largest
components of $x + h$ are $\{x_j + h_j\}_{j \in [k]}$, we have
$$
\|x+h\|_w = \|(x+h)_{[k]}\|_{w_{[k]}} + \|(x+h)_{[k+1:p]}\|_{w_{[k+1:p]}}.
$$

Now, take $h \in \R^p$ such that $\|h\|_\infty < e$ and $h_{k+1} =
\dots = h_p = 0$. Using the above identity and the definition of
$F_w(x)$, one may deduce that
\begin{align*}
\|(x+h)_{[k]}\|_{w_{[k]}} & =  \|x+h\|_w- \|x_{[k+1:p]}\|_{w_{[k+1:p]}} \\
& \geq \|x\|_w + s'h - \|x_{[k+1:p]}\|_{w_{[k+1:p]}} = 
\|x_{[k]}\|_{w_{[k]}} + \sum_{j=1}^k s_j h_j. 
\end{align*}
We therefore obtain that
$$
\|x_{[k]} + h\|_{w_{[k]}} \ge \|x_{[k]}\|_{w_{[k]}} + s_{[k]}'h
$$
for all $h \in \R^k$ satisfying $\|h\|_\infty < e$. By
(\ref{eq:subdiff-local}), we conclude $s_{[k]} \in
F_{w_{[k]}}(x_{[k]})$. To show that $s_{[k+1:p]} \in
F_{w_{[k+1:p]}}(x_{[k+1:p]})$, one can proceed in a similar manner.


($\supseteq$) For $s\in F_{w_{[k]}}(x_{[k]}) \times
F_{w_{[k+1:p]}}(x_{[k+1:p]})$, we clearly have
$$
s'x = \sum_{i=1}^k s_i x_i + \sum_{i=k+1}^p s_i x_i =
\|x_{[k]}\|_{w_{[k]}} + \|x_{[k+1:p]}\|_{w_{[k+1:p]}} = \|x\|_w,
$$
so that $s \in F_w(x)$ follows.


\emph{4)} For $x \in \R^p$ with $x_1 = \dots = x_{k_1} > \dots  >
x_{k_l + 1} = \dots = x_p$, $\patt(x)$ is clearly given by
$$
\begin{cases}
\patt(x)_1 = \dots = \patt(x)_{k_1} = l+1 > \dots > \patt(x)_{k_l+1} = \dots = \patt(x)_p = 1
& \text{if } x_p > 0 \\
\patt(x)_1 = \dots = \patt(x))_{k_1} = l > \dots > \patt(x)_{k_l+1} = \dots = \patt(x)_p = 0
& \text{if } x_p = 0.
\end{cases}
$$
According to 1), 2) and 3), it is clear that
$$
F_w(x) = F_w(\patt(x)) = 
\begin{cases}
P_{w_{[k_1]}} \times \dots \times P_{w_{[k_{l-1}+1:k_l]}} \times P_{w_{[k_l+1:p]}} 
& { if} x_p > 0 \\
P_{w_{[k_1]}} \times \dots \times P_{w_{[k_{l-1}+1:k_l]}} \times P^\pm_{w_{[k_l+1:p]}}
& { if} x_p = 0.
\end{cases}
$$
Since the codimension of a permutahedron is equal to $1$
\citep[see][]{MaesKappen92,Simion97}, the one of signed permutahedron is
equal to $0$, and since the (co-)dimensions of the individual (sign)
permutahedra can simply be added up, we have
$\codim\left(F_w(x)\right) = \|\patt(x)\|_\infty$.
\end{proof}

Proposition~\ref{prop:subdiff-SLOPE} lays the groundwork by
essentially proving Theorem~\ref{thm:SLOPE-pattern} for all SLOPE patterns
with non-negative and non-decreasing components. We denote this set of
patterns by $\mM_p^{\geq,+}$, given by
$$
\mM_p^{\geq,+} = \{m \in \mM_p : m_1 \geq \dots \geq m_p \geq 0\}.
$$

In order to extend this proposition to all SLOPE patterns in $\mM_p$, we
introduce the following group of linear transformations.
\begin{definition}
Let $\sigma \in \{-1,1\}^p$, let $\pi \in \mS_p$. We define the map
$$
\phi_{\sigma,\pi}: x \in \R^p \mapsto (\sigma_1 x_{\pi(1)},\dots,\sigma_p x_{\pi(p)})'
$$
and denote by $\mG = \{\phi_{\sigma,\pi}: \sigma \in \{-1,1\}^p, \pi \in
\mS_p\}$.
\end{definition}

The set $\mG$ is a finite sub-group of the group of orthogonal
transformations on $\R^p$. We list a number of straight-forward
properties of $\mG$ in the following lemma.

\begin{lemma} \label{lem:trans-prop}
Let $x,v \in \R^p$, $\phi \in \mG$, and let $\sigma \in \{-1,1\}^p$ and
$\pi \in \mS_p$. Then the following holds.

\begin{enumerate}

\item $x'v = \phi(x)'\phi(v)$

\item $\|x\|_w = \|\phi(x)\|_w$

\item $\|x\|_\infty = \|\phi(x)\|_\infty$

\item $\phi(\mM_p) = \mM_p$ and $\phi(P_w^\pm) = P_w^\pm$

\item $\patt(\phi(x)) = \phi(\patt(x))$

\item $\phi_{\sigma,\pi}^{-1} = \phi_{\sigma,\pi^{-1}} \in \mG$ 

\item If, for $m \in \mM_p$, $|m_{\pi(1)}| \geq \dots \geq
|m_{\pi(p)}|$ and $\sigma_j m_{\pi(j)} = |m_{\pi(j)}|$ for all $j \in
[p]$, then $\phi_{\sigma,\pi}(m)\in \mM_p^{\geq+}$.

\end{enumerate}

\end{lemma}

\begin{lemma} \label{lem:subdiff-trans} 
Let $\phi \in \mG$ and $x \in \R^p$. We then have
$$
\phi^{-1}\left(F_w(\phi(x))\right) = F_w(x) \; \text{ and } \; 
F_w(\phi(x)) = \phi\left(F_w(x)\right).
$$
\end{lemma}

\begin{proof} 
The two statements are equivalent, we show the second one. Let $s \in
P_w^\pm$. Then
\begin{align*}
s \in F_w(\phi(x)) & \iff s'\phi(x) = \|\phi(x)\|_w \iff \phi^{-1}(s)'x = \|x\|_w \\
& \iff \phi^{-1}(s) \in F_w(x) \iff s \in \phi(F_w(x))
\end{align*}
by Proposition~\ref{prop:subdiff-norm} and Lemma~\ref{lem:trans-prop}.
\end{proof}

We are now equipped to prove Theorem~\ref{thm:SLOPE-pattern}.

\begin{proof}[Proof of Theorem~\ref{thm:SLOPE-pattern}] \leavevmode

We start by proving \emph{1)} and \emph{2)} before showing that the
map is a bijection.

\emph{1)} Let $m \in \mM_p$ and let $\phi \in \mG$ such that $\phi(m)
\in \mM_p^{\geq,+}$. According to Lemma~\ref{lem:subdiff-trans}, and
because $\phi$ is an isomorphism on $\R^p$, we have
$$
\codim(F_w(m)) = \codim\left(\phi^{-1}\left(F_w(\phi(m))\right)\right)
= \codim\left(F_w(\phi(m))\right) = \|\phi(m)\|_\infty = \|m\|_\infty.
$$

\emph{2)} Let $x \in \R^p$ and let $\phi \in \mG$ such that $\phi(x)_1
\geq \dots \geq \phi(x)_p \geq 0$. According to
Lemma~\ref{lem:subdiff-trans} and
Proposition~\ref{prop:subdiff-SLOPE}, the following equalities hold
$$
F_w(x) = \phi^{-1}\left(F_w(\phi(x))\right) 
= \phi^{-1}\left(F_w\left(\patt(\phi(x))\right)\right) 
= \phi^{-1}\left(F_w\left(\phi\left(\patt(x)\right)\right)\right)
= F_w(\patt(x)).
$$

We now show that the mapping under consideration is indeed a bijection
between $\mM_p$ and $\mF_0$.


(surjection) According to Proposition~\ref{prop:subdiff-face}, a
non-empty face of $P_w^\pm$ can be expressed as $F_w(x)$ for some $x
\in \R^p$. According to \emph{2)} above, we have $F_w(x) =
F_w(\patt(x))$ for $\patt(x) \in \mM_p$.


(injection) Note that Proposition~\ref{prop:subdiff-SLOPE} shows that
the mapping is injective on $\mM_p^{\geq,+}$. To prove that it remains
injective on all of $\mM_p$, we show that $|\mM_p| \leq |\mF_0|$. For
this, we need several definitions. For $m \in \mM_p$, let
$\stab_\mG(m) = \{\phi \in \mG: \phi(m) = m\}$ and $\orb_\mG(m) =
\{\phi(m) : \phi \in \mG\}$, the stabilizer and orbit of $m$,
respectively, with respect to $\mG$. For $m\in \mM_p$, there exists
$\phi \in \mG$ such that $\phi(m) \in \mM_p^{\geq,+}$. Therefore, the
orbit-stabilizer formula \citep{Artin11}[Proposition~6.8.4] gives
$$
\mM_p \; = \! \!\! \bigcup_{m\in \mM_p^{\ge,+}} \!\! \orb_\mG(m) \implies |\mM_p| \leq
\sum_{m \in \mM_p^{\geq,+}} |\orb_\mG(m)| = \sum_{m \in \mM_p^{\geq,+}}
\frac{|\mG|}{|\stab_\mG(m)|}.
$$
We also look at stabilizer and orbit when $\mG$ operates on $\mF_0$.
For a face $F \in \mF_0$, let $\stab_\mG(F) = \{\phi \in \mG : \phi(F)
= F\}$ and $\orb_\mG(F) = \{\phi(F) : \phi \in \mG\}$. We first show
that if $\orb_\mG(F_w(m)) \cap \orb_\mG(F_w(\tilde m)) \neq \emptyset$
for some $m, \tilde m \in \mM_p^{\geq,+}$, $m = \tilde m$ follows. Let
us assume that $F_w(\tilde m) = \phi(F_w(m))$ for some $\phi \in \mG$.
Note that $\phi(F_w(m)) = F_w(\phi(m))$ by
Lemma~\ref{lem:subdiff-trans}. Since $w \in F_w(m)$ and $w \in
F_w(\tilde m) = F_w(\phi(m))$, we have
$$
w'm = \|m\|_w =  \|\phi(m)\|_w = w'\phi(m),
$$ 
where the first equality holds since $m \in \mM_p^{\geq,+}$, the
second equality holds by Lemma~\ref{lem:trans-prop} and the last
equality holds since $m\in F_w(\phi(m))$.  Now, if $\phi(m) \neq m$,
$w'\phi(m)<\|m\|_w$ follows since the components of $w$ are positive
and strictly decreasing. But that would contradict the above, so
$\phi(m) = m$ must hold. Consequently, $F_w(\tilde m) = F_w(m)$, which
in turn implies $\tilde m = m$ by
Proposition~\ref{prop:subdiff-SLOPE}.

Now, let $m \in \mM^{\geq,+}_p$ and let us show that $\stab_\mG(m) =
\stab_\mG(F_w(m))$. The inclusion $\stab_\mG(m) \subseteq
\stab_\mG(F_w(m))$ immediately follows from
\begin{align*}
\phi \in \stab_\mG(m) & \implies F_w(m) = \phi^{-1}(F_w(\phi(m))) =
\phi^{-1}(F_w(m)) \implies \phi(F_w(m)) = F_w(m) \\
& \implies \phi \in \stab_\mG(F_w(m)).
\end{align*}
To show $\stab_\mG(F_w(m)) \subseteq \stab_\mG(m)$, let $\phi \in
\stab_\mG(F_w(m))$ and note that $F_w(m) = \phi(F_w(m)) =
F_w(\phi(m))$. Since $m \in \mM^{\geq,+}_p$, this implies that $w \in
F_w(m) = F_w(\phi(m))$, so that the same reasoning as above yields $m
= \phi(m)$ and $\phi \in \stab_\mG(m)$.

To conclude, note that since the orbits $\orb_\mG(F_w(m))$ with $m\in
\mM_p^{\geq,+}$ are disjoint, and since $\stab_\mG(m)
=\stab_\mG(F_w(m))$, we may deduce that
$$
|\mM_p| \leq \sum_{m \in \mM_p^{\geq,+}} \frac{|\mG|}{|\stab_\mG(F_w(m))|} 
= \sum_{m \in \mM_p^{\geq,+}} \big|\orb_\mG(F_w(m))\big| 
= \Big|\!\!\!\!\bigcup_{m \in \mM_p^{\geq,+}} \!\! \orb_\mG\left(F_w(m)\right)\Big| 
\leq |\mF_0|.
$$
\end{proof}

\subsection{Proof of Theorem~\ref{thm:SLOPE-acc}} 
\label{app:proof-SLOPE-acc}

The following lemma generalizes Proposition~4.1 from \cite{Gilbert17}
that is stated for the $\ell_1$-norm to an arbitrary norm. This lemma
is used in the proof of both Theorem~\ref{thm:SLOPE-acc} and
Theorem~\ref{thm:LASSO_BP-acc}.

\begin{lemma} \label{lem:analytic}
Let $s \in \R^p$ and $\|.\|$ be a norm on $\R^p$. The vector space
$\row(X)$ intersects $\dnorm(s)$ if and only if the following holds.
\begin{equation} \label{eq:analytic}
Xb = Xs \implies \|b\| \geq \|s\|
\end{equation}
\end{lemma}

\begin{proof}
Consider the function $f_s : \R^p \to \{0,\infty\}$ given by
$$
f_s(b) =
\begin{cases} 0 & Xb = Xs \\
\infty & \text{else}. 
\end{cases}
$$
Then (\ref{eq:analytic}) holds for $b$ if and only if $s$ is a
minimizer of the function $b \mapsto \|b\| + f_s(b)$. Since we have
$\partial_{f_s}(b) = \row(X)$ whenever $Xb = Xs$, we can deduce that
the implication (\ref{eq:analytic}) occurs if and only if
$$
0 \in \row(X) + \dnorm(s) \iff \row(X) \cap \dnorm(s) \neq \emptyset.
$$
\end{proof}


\begin{proof}[Proof of Theorem~\ref{thm:SLOPE-acc}] 
($\implies$) If $m$ is an accessible SLOPE pattern, then
$$
\exists y \in \R^n, \exists \hat\beta \in \Sslope(y) \text{ such that
} \patt(\hat\beta) = m.
$$
By Theorem~\ref{thm:SLOPE-pattern}, we may deduce that
$\dnormx{w}(\hat\beta) = F_w(\hat\beta) = F_w(m)$. Consequently,
$$
0 \in X'(X\hat\beta - y) + \dnormx{w}(\hat\beta)  \implies X'(y - X\hat\beta) \in F_w(m).
$$ 
Therefore, $\row(X)$ intersects $F_w(m)$ (geometric characterization),
or, equivalently, by Lemma~\ref{lem:analytic}, whenever $Xb = Xm$ we
have $\|b\|_w \geq \|m\|_w$ (analytic characterization).


($\impliedby$) If $\row(X)$ intersects the face $F_w(m)$ (geometric
characterization), or, equivalently, whenever $Xb = Xm$ we have
$\|b\|_w \geq \|m\|_w$ (analytic characterization), there exists $z
\in \R^n$ such that $X'z = f \in F_w(m)$. We set $y = z + Xm$ and show
that $m \in \Sslope(y)$. We have
$$
\|X'(y - Xm)\|_w^* = \|f\|_w^* \leq 1 \text{ and } m'X'(y-Xm) = m'f = \|m\|_w,
$$
which, by Proposition~\ref{prop:ch-pen}, yields $m \in \Sslope(y)$.
\end{proof}

\subsection{Proof of Proposition~\ref{prop:SLOPE-As}} 
\label{app:proof-SLOPE-As}

\begin{proof}
By Theorem~\ref{thm:SLOPE-acc}, we know that
$$
m \in \mM_p \text{ is accessible} \iff \row(X) \cap F_w(m) \neq
\emptyset \iff \exists f \in \R^n: X'f \in F_w(m) \iff f \in N_w(m),
$$
which proves the first statement. Now, let $y = f + Xb$, where $f \in
N_w(m)$ and $b \in \R^p$ such that $\patt(b) = m$. Note that
$$
\|X'(y - Xb)\|^*_w = \|X'f\|^*_w \leq 1 \text{ and } b'X'(y - Xb) = b'X'f = \|b\|^*_w, 
$$
where the first inequality holds since $X'f \in F_w(m)$, a face of
$P_w^\pm$, and the latter one by applying
Proposition~\ref{prop:subdiff-face} after noticing that $X'f \in
F_w(m) = F_w(b) = \dnormx{w}(b)$ by Theorem~\ref{thm:SLOPE-pattern}.
Proposition~\ref{prop:ch-pen} then yields $b \in \Sslope(y)$, so that
$y \in A_w(m)$.

Conversely, let $y \in A_w(m)$ and let $\hat\beta \in \Sslope(y)$ so
that $\patt(\hat\beta) = m$. Then $y - X\hat\beta \in N_w(m)$ since
by Proposition~\ref{prop:ch-pen}, we have
$$ 
X'(y - X\hat\beta) \in \dnormx{w}(\hat\beta) = F_w(m),
$$ 
where the last equality holds by Theorem~\ref{thm:SLOPE-pattern}.
\end{proof}

\subsection{Proof of Proposition~\ref{prop:convexnullset}} 
\label{app:proof-convexnullset}

\begin{proof}
Note that by Proposition~\ref{prop:ch-pen} we have that $\hat\beta \in
\Spen(y)$ if and only if we have
$$
\|X'(y - X\hat\beta)\|^* \leq 1 \text{ \rm and } \hat\beta'(y -
X\hat\beta) = \|\hat\beta\|.
$$
Consequently, when $\|X'u\|^* \leq 1$ it is clear that $0 \in
\Spen(u)$ implying that $\Spen(u) = \{0\}$ as all elements of
$\Spen(u)$ must have the same norm. Now, let $u \in A_\emptyset$ and
remember that $\hat u= y - X\hat\beta$. The following inequality
$$
(y-\hat u)'(u-\hat u) = \underbrace{\hat\beta'X'u}_{ \leq \|\hat\beta\|} -
\underbrace{\hat\beta'X'(y - X\hat\beta)}_{ = \|\hat \beta\|} \leq 0
$$
shows that, indeed, $\hat u$ is the projection of $y$ onto the convex
null set $A_\emptyset$.
\end{proof}

\subsection{Proof of Theorem~\ref{thm:LASSO_BP-acc}}
\label{app:proof-LASSO_BP-acc}

\begin{proof} ($\implies$) 
Let $\sigma$ be an accessible sign vector for LASSO. Then there exists
$y \in \R^n$ and $\hat\beta \in \SlassoL(y)$ such that
$\sign(\hat\beta) = \sigma$. According to the characterization of
LASSO minimizers in (\ref{eq:ch-lasso}), by setting $z = (y -
X\hat\beta)/\lambda$, one may deduce that $X'z \in F_1(\sigma)$. If
$\sigma$ is an accessible sign vector for BP, there exists $y \in
\col(X)$ and $\hat\beta \in \Sbp(y)$ with $\sign(\hat\beta) = \sigma$.
According to the characterization of BP minimizers in (\ref{eq:chBP}),
there exisits $z \in \R^n$ such that $X'z \in F_1(\sigma)$. Therefore,
$\row(X)$ intersects $F_1(\sigma) = \dnormx{1}(\sigma)$ (geometric
characterization), or, equivalently, by Lemma~\ref{lem:analytic},
whenever $Xb = X\sigma$, we have $\|b\|_1 \geq \|\sigma\|_1$ (analytic
characterization).


($\impliedby$) If $\row(X)$ intersects the face $F_1(\sigma)$
(geometric characterization) or, equivalently, if $Xb = X\sigma$
implies $\|b\|_1 \geq \|\sigma\|_1$ (analytic characterization), then
there exists $f \in F_1(\sigma)$ and $z \in \R^n$ such that $X'z = f$.
Note that $j \in \supp(\sigma)$ implies that $f_j = \sigma_j =
\sign(\sigma_j)$. Set $y = \lambda z + X\sigma$  We show that $\sigma
\in \SlassoL(y)$. We have
$$
\begin{cases}
\|X'(y - X\sigma)\|_\infty = \lambda \|X'z\|_\infty \leq \lambda,\\
X_j'(y - X\sigma) = \lambda X_j'z = \lambda f_j = \lambda \sigma_j =
\lambda\sign(\sigma_j) & \forall j \in \supp(\sigma),
\end{cases}
$$
so that according to the characterization of LASSO minimizers in
(\ref{eq:ch-lasso}), we have $\sigma \in \SlassoL(y)$, implying that
$\sigma$ is accessible for LASSO. For BP, set $y = X\sigma$ and note
that, according to the characterization of BP minimizers in
(\ref{eq:chBP}), $\sigma \in \Sbp(y)$, implying that $\sigma$ is also
accessible for BP.
\end{proof}

\begin{corollary} \label{cor:BP-acc}
Let $X \in \R^{n \times p}$, $\sigma \in \{-1,0,1\}^p$ and assume that
$\hat\beta$ is the unique solution to the BP problem $\Sbp(y)$ with $y
= X\sigma$. Let $\tilde\beta \in \R^p$ satisfy $\|\tilde\beta -
\hat\beta\|_\infty < 1/2$. We then have that
$$
\sigma \text{ is accessible } \iff \sign(\round(\tilde \beta)) = \sigma,
$$
where $\round(.)$ rounds componentwise to the nearest integer.
\end{corollary}

\begin{proof}
($\implies$) If $\sigma$ is accessible, by the analytic
characterization in Theorem~\ref{thm:LASSO_BP-acc}, $\hat\beta =
\sigma$. Since $\|\tilde\beta - \sigma\|_\infty < 1/2$, we get
$\sign(\round(\tilde\beta)) = \round(\tilde\beta) = \sigma$.

($\impliedby$) If $\sigma$ is not accessible, we have $\sign(\hat\beta)
\neq \sigma$. Using $\|\tilde\beta - \hat\beta\|_\infty < 1/2$, we can show that
$$
F_1(\sign(\hat\beta)) = \dnormx{1}(\sign(\hat\beta)) \subseteq
\dnormx{1}(\sign(\round(\tilde\beta))) =
F_1(\sign(\round(\tilde\beta))).
$$
Since $\row(X)$ intersects $F_1(\sign(\hat\beta))$ by the geometric
characterization in Theorem~\ref{thm:LASSO_BP-acc},
$\sign(\round(\tilde\beta))$ is accessible. But then
$\sign(\round(\tilde \beta)) \neq \sigma$ must hold.
\end{proof}


\end{document}